\documentclass[ envcountsect, envcountsame]{svjour3}
\usepackage[utf8x]{inputenc}
\usepackage[english]{babel}
\usepackage{amsmath}
\usepackage{amssymb}
\usepackage{xcolor}
\usepackage[colorlinks=false,linkbordercolor=red,citebordercolor=blue]{hyperref}
\usepackage{tabularx}
\usepackage{comment}
\usepackage{graphicx}
\usepackage{epsfig}
\usepackage{caption}
\usepackage{subcaption}
\usepackage{ntheorem}
\usepackage{enumitem}

\usepackage{pgfplots}
\usepackage{tikz}
\usetikzlibrary{shapes, arrows}
\tikzstyle{block} = [draw, fill=white, rectangle, minimum height=0em, minimum width=0em]
\tikzstyle{output} = [coordinate]
\tikzstyle{input} = [coordinate]

\smartqed
\newtheorem{Theorem}{Theorem}[section]

\newtheorem{Definition}[Theorem]{Definition}

\newtheorem{Corollary}[Theorem]{Corollary}

\theorembodyfont{\normalfont}
\newtheorem{Example}[Theorem]{Example}
\newtheorem{Remark}[Theorem]{Remark}

\newcommand{\setdef}[2]{\left\{\, #1\, \left|\, \vphantom{#1} #2 \right.\right\}}

\newcommand{\ddt}{\tfrac{\text{\normalfont d}}{\text{\normalfont d}t}}
\newcommand{\dt}{\text{\normalfont d}t}
\newcommand{\ds}{\text{\normalfont d}s}
\newcommand{\R}{\mathbb{R}}
\newcommand{\C}{\mathbb{C}}
\newcommand{\N}{\mathbb{N}}
\newcommand{\cB}{\mathcal{B}}
\newcommand{\cC}{\mathcal{C}}
\newcommand{\cD}{\mathcal{D}}
\newcommand{\cF}{\mathcal{F}}
\newcommand{\cL}{\mathcal{L}}
\newcommand{\cN}{\mathcal{N}}

\newcommand{\cP}{\mathcal{P}}
\newcommand{\cT}{\mathcal{T}}

\newcommand{\cW}{\mathcal{W}}

\newcommand{\vp}{\varphi}
\newcommand{\ve}{\varepsilon}
\newcommand{\rp}{\mathbb{R}_{\geq 0}}

{\left(\begin{smallmatrix}}
{\end{smallmatrix}\right)}

\newenvironment{smallbmatrix}
{\left[\begin{smallmatrix}}
{\end{smallmatrix}\right]}

\DeclareMathOperator{\im}{im}

\DeclareMathOperator{\Gl}{\mathbf{Gl}}
\DeclareMathOperator{\rk}{\rm rk}

\DeclareMathOperator{\esssup}{\rm ess\,sup}

\def\red#1{\textcolor[rgb]{0.60,0.10,0.07}{#1}}
\def\blue#1{\textcolor[rgb]{0.00,0.20,0.70}{#1}}
\def\green#1{\textcolor[rgb]{0.00,0.50,0.50}{#1}}

\begin{document}
\title{ Output feedback control with prescribed performance via funnel pre-compensator
\thanks{This work was supported by the German Research Foundation
(Deutsche Forschungsgemeinschaft) via the grant BE 6263/1-1.}
}

\author{Lukas Lanza}
\institute{Lukas Lanza \at
              \email{lanza@math.upb.de}           
}

\date{Received: date / Accepted: date}

\maketitle

\begin{abstract}
We study output reference tracking of systems with high relative degree via output feedback only; this is, tracking where the output derivatives are unknown.
To this end, we prove that the conjunction of the funnel pre-compensator with a minimum phase system of arbitrary relative degree yields a system of the same relative degree which is minimum phase as well. 
The error between the original system's output and the pre-compensator's output evolves within a prescribed performance funnel; and moreover, the derivatives of the funnel pre-compensator's output are known explicitly. 
Therefore, output reference tracking with prescribed transient behaviour of the tracking error is possible without knowledge of the derivatives of the original system's output; via funnel control schemes for instance.
\keywords{Funnel pre-compensator \and output feedback \and funnel control \and output reference tracking \and nonlinear systems}
\end{abstract}

\section*{Nomenclature}
Throughout the present article we use the following notation, where $I \subseteq \R$ denotes an interval
\ \\
\begin{tabularx}{\textwidth}{rX}
$\N$& the set of positive integers, \\
$\rp$, $\C_-$ & the sets $:= [0, \infty)$, $\setdef{ \mu \in \C}{ \text{Re}(\mu) < 0}$, respectively, \\
$A \in \R^{n \times m}$ & the matrix~$A$ is in the set of real $n \times m$ matrices, \\
$\Gl_n(\R)$ & the group of invertible matrices in~$\R^{n \times n}$,  \\
$A > 0$ &$:= \iff x^\top A x > 0$ for all $x \in \R^n \setminus\{0\}$, the matrix~$A \in \R^{n \times n}$ is positive definite \\
$\sigma(A)$ & $:= \{\lambda \in \C \ \vline \ \det(A-\lambda I_n)=0 \}$ spectrum of the matrix $A \in \R^{n \times n}$, \\
$\lambda_{\rm max}(A), \lambda_{\rm min}(A) $ & the largest and the smallest eigenvalue of a matrix $A \in \R^{n \times n}$ with~$\sigma(A) \subseteq \R$, respectively, \\
$\| x \| $ & $:=\sqrt{x^\top x} $ Euclidean norm of  $x \in \R^n$,  \\
$\| A \|$ & $:= \max_{\|x\|=1} \|Ax\|$ spectral norm of $A \in \R^{m \times n}$, \\
$\cL_{\rm loc}^\infty (I \to \R^p)$ & set of locally essentially bounded functions $ f: I \to \R^p$ ,  \\
$\cL^\infty (I \to \R^p)$ & set of essentially bounded functions $ f: I \to \R^p$ ,  \\
$\| f \|_{\infty} $ & $:= \esssup_{t \in I} \|f(t)\| $ norm of $f \in \cL^\infty(I \to \R^p)$ , \\
$\cW^{k,\infty}(I \to \R^p)$ & set of $k$-times weakly differentiable functions\\& $ f : I \to \R^p$ such that $f,\ldots,f^{(k)} \in \cL^\infty(I \to \R^p)$, \\
$\cC^k( I \to \R^p) $ & set of $k$-times continuously differentiable functions \\& $f : I \to \R^p$, $\cC(I\to \R^p) = \cC^0(I\to \R^p)$, \\
$f|_{J}$ & the restriction of $f : I \to \R^n$ to $J \subseteq I$, $I$ an interval, \\
a. a. & almost all.
\end{tabularx}

\ \\
For later use, we recall the Kronecker product of two matrices~$L \in \R^{l \times m}$ and $K = \left( k_{ij} \right)_{i=1,\ldots,k;j=1,\ldots,n } \in \R^{k \times n}$ 
\begin{equation*}
K \otimes L := \begin{bmatrix} k_{11} L & \cdots & k_{1n} L \\ \vdots & \ddots & \vdots \\ k_{k1} L & \cdots & k_{kn} L\end{bmatrix} \in \R^{kl \times nm}.
\end{equation*}

\section{Introduction}
In the present article we elaborate on the so called \textit{funnel pre-compensator}, first proposed in~\cite{BergReis18b}.
The funnel pre-compensator is a simple adaptive dynamical system of high-gain type which receives signals from a certain class of signals specified later, and has an output which approximates the input signal in the sense that the error between the input signal and the pre-compensator's output evolves within a prescribed performance funnel; 
moreover, the derivatives of the pre-compensator's output are known explicitly.
Comparing the preprint~\cite{BergReis16app} and the work~\cite{BergReis18b} it is clear that the funnel pre-compensator was inspired by the concept of high-gain observers (mainly inspired by the adaptive high-gain observer proposed in~\cite{BullIlch98}); for detailed literature on high-gain observers see~\cite{EsfaKhal87,KhalSabe87,SabeSann90,Torn88} and the survey~\cite{KhalPral14} as well as the references therein, respectively.
For a discussion and detailed comparison of some properties of high-gain observers and the funnel pre-compensator see~\cite{BergReis18b}.

\ \\ 
Although there is plenty of properly working high-gain based feedback controller with prescribed error performance, for \textit{funnel control} schemes see e.g.~\cite{IlchRyan02b,BergLe18a},
the recent work~\cite{BergIlch21} or the construction of a bang-bang funnel controller cf.~\cite{LibeTren13b}, and for \textit{prescribed performance controller} see~\cite{BechRovi09,BechRovi14},
 all suffer from the problem that the output signal's derivatives (funnel control) or the full state (prescribed performance controller) have to be available to the control scheme. 
For funnel control this means, if the output's derivatives are not available from measurement, the output measurement has to be differentiated which is an ill-posed problem, see e.g.~\cite[Sec. 1.4.4]{Hack12}.
Prominent ideas in the literature to handle this topic are so called \textit{backstepping} procedures, see e.g.~\cite{IlchRyan06b,IlchRyan07} in conjunction with an input filter. 
However, the backstepping procedure typically involves high powers of a ``large-valued" gain function, which causes numerical issues and leads to impractical performances, see~\cite[Sec. 4.4.3]{Hack12}.
Another approach to solve an \textit{arbitrary good transient and steady-state response problem} for linear minimum phase systems with arbitrary relative degree is presented in~\cite{MillDavi91}.
The proposed controller involves an internal compensator scheme of LTI type which allows to achieve an arbitrary small error within an arbitrary short time  receiving the systems's output and the reference signal only. 
Although this control scheme has a number of advantageous features such as noise tolerance and applicability to systems with unknown relative degree to name but two (see also the survey~\cite{Ilch91}), it is an adaptive scheme with a monotonically non-decreasing gain and involves a (piecewise constant) switching function where the switching times are determined in a two phase scheme of rather high complexity.
In the works~\cite{ChowKhal19,DimaBech20} approaches to realize output tracking with prescribed error behaviour via output feedback only are presented.
In~\cite{ChowKhal19} single-input single-output systems of known arbitrary relative degree with bounded input bounded state stable internal dynamics are under consideration. 
The control scheme involves higher derivatives of the output which are approximated via a high-gain observer.
With this, tracking via output feedback can be realized. However, in this setting knowledge of the control coefficient is required and hence the particular control scheme is - in contrast to standard funnel control schemes - not model free. 
In~\cite{DimaBech20} an extension of the prescribed performance controller~\cite{BechRovi14} is used to achieve output tracking with prescribed error performance of unknown nonlinear multi-input multi-output systems with known vector relative degree.
A high-gain observer scheme is used to make the required derivatives available. 
Since the control schemes in~\cite{ChowKhal19,DimaBech20} involve high-gain observers both suffer from the problem of proper initializing, i.e., the high-gain parameters are to be predetermined appropriately; however, it is not clear how to choose these parameters appropriately in advance.
%
In~\cite{LiuSu21} an output feedback funnel control scheme is developed which achieves output tracking with prescribed transient behaviour for a class of nonlinear single-input single-output systems where the nonlinearity is a function of time and the output variable only. In particular, the problem of choosing parameters appropriately in advance is circumvented.

\ \\
As mentioned above the derivatives of the pre-compensator's output are known explicitly, and hence the aforesaid gives rise to the idea that the funnel pre-compensator scheme proposed in~\cite{BergReis18b} can help resolving the long-standing problem of adaptive feedback control with prescribed error performance of nonlinear systems with relative degree higher than one with unknown output derivatives.
In order to resolve this problem, in the present article we prove that the application of a cascade of funnel pre-compensators to a minimum phase system of arbitrary relative degree yields a system of the same relative degree, which is minimum phase as well. 
In particular, the derivatives of the pre-compensator's output are known explicitly.
Therefore, output reference tracking with prescribed transient behaviour using well known funnel control schemes for systems of arbitrary (possibly high) relative degree, as for instance from~\cite{BergLe18a} or the recent work~\cite{BergIlch21}, is possible without knowledge of the system's output derivatives. 
In particular, the tracking error between the original system's output and the desired reference trajectory evolves within a prescribed performance funnel.
For systems of relative degree two this was shown in~\cite{BergReis18b} and this result was used for funnel control in~\cite{BergReis18a}, but for arbitrary relative degree~$r\in\N$ this remained an open problem which we solve in the present paper.

\ \\
Before we recall and investigate the funnel pre-compensator introduced in~\cite{BergReis18b} we highlight that, contrary to most approaches, the funnel pre-compensator does not necessarily receive signals~$u$ and~$y$ which are input and output of a dynamical system or a corresponding plant, 
but, defining $\cL \cW_{m}^{r,\infty} := \cL_{\rm loc}^\infty(\rp \to \R^m) \times \cW_{\rm loc}^{r,\infty}(\rp \to \R^{m}) $, we consider signals $u$ and $y$ belonging to the large set
\begin{equation*}
\cP_r := 
\left\{ 
(u,y) \in \cL \cW_{m}^{r,\infty}  \ \vline \ 
\begin{array}{l}
\exists\, \Gamma \in \cC^1(\rp \to \R^{m \times m}): \\ 
\Gamma y^{(r-1)} \in \cL^\infty(\rp \to \R^m), \\
\ddt(\Gamma y^{(r-1)}) - u \in \cL^\infty(\rp \to \R^m)
\end{array}
\right\}.
\end{equation*}
We emphasize that it is not assumed to know the matrix valued function~$\Gamma$. 
Only knowledge of the signals~$u$ and~$y$ and the number~$ r \in \N$ is assumed.
It is self evident that the signals~$u$ and~$y$ can be input and output of a corresponding plant, respectively; for an example see~\cite{BergReis18b};
however, the signal set~$\cP_r$ allows for a much larger class of dynamical systems, cf.~\cite{BergLe18a} and the works~\cite{BergIlch14,IlchRyan02b}.

\ \\
The present article is organized as follows.
In Section~\ref{Sec:The-funnel-pre-compensator} we recall the concept of the funnel pre-compensator first introduced in~\cite{BergReis18b}, and recapitulate the respective results we will work with.
Section~\ref{Sec:main-result-FPC-min-phase} contains the main result of the present article. 
After introducing the system class under consideration in subsection~\ref{Sec:System-class} and establishing the set of feasible design parameters of the pre-compensator in subsection~\ref{Sec:Parameter}, 
in subsection~\ref{Sec:FPC} we state that the application of a cascade of funnel pre-compensators to a minimum phase system with arbitrary relative degree~$r \in \N$ leads to a system of same relative degree which is minimum phase as well, and moreover, the first~$r-1$ derivatives of the pre-compensator's output are known explicitly; this is, in subsection~\ref{Sec:FPC} we  present the extension of~\cite[Thm. 2]{BergReis18b} to arbitrary relative degree.
The proof of this result is relegated to the Appendix.
In Section~\ref{Sec:Tracking} we turn towards an application of the funnel pre-compensator, namely output tracking via output feedback only.
We show that with the aid of the funnel pre-compensator output tracking with prescribed transient behaviour of the tracking error with unknown output derivatives is possible via funnel control techniques.
In Section~\ref{Sec:Simulations} we provide numerical simulations illustrating the findings from Section~\ref{Sec:Tracking}.

\section{The funnel pre-compensator} \label{Sec:The-funnel-pre-compensator}
In order to incorporate the main result 
of the present article into the context of the funnel pre-compensator proposed in~\cite{BergReis18b}, 
we briefly recall the respective results.
The funnel pre-compensator is a pre-compensator of high-gain type in the spirit of funnel control; for details concerning funnel control see the 
works~\cite{IlchRyan02b,IlchRyan08,BergLe18a}, the recent work~\cite{BergIlch21} and the references therein, respectively.
The funnel pre-compensator~\eqref{eq:FPC} is a dynamical system receiving signals ${(u,y) \in \cP_r}$, for some~$r \in \N$, and giving~$z$ as an output, the first derivative of the latter is known exactly. 
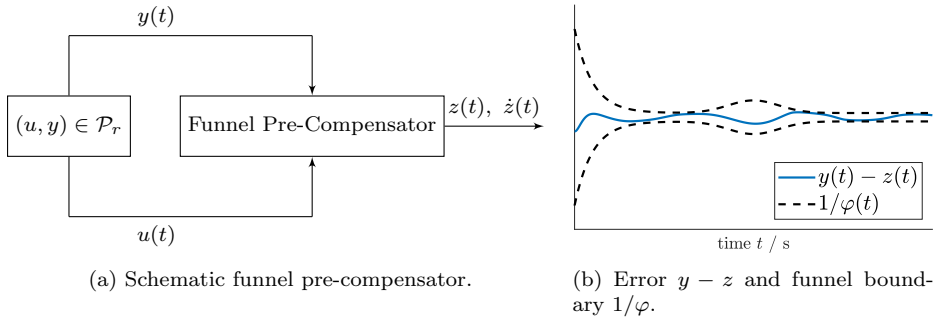
\begin{figure}[h!]
\begin{subfigure}[t]{0.59\textwidth}
\begin{tikzpicture} [auto, node distance=2cm,>=latex']
\def\hoch{0.8cm};
\def\breit{0cm};
\def\distu{1.2cm};
\def\dista{1.2cm};
\node [block, minimum width = \breit, minimum height = \hoch,] (Signal) { $(u,y) \in \cP_r$ };
\node [block, minimum width = \breit, minimum height = \hoch, right of=Signal, node distance=3.2cm] (FP1) {Funnel Pre-Compensator};
\node [output, right of=FP1, node distance=3.05cm] (z) {  };
\draw[->] (FP1) -- node[name=z]{$z(t), \ \dot z(t)$} (z) ;
\node [output, above of=Signal, node distance = \dista] (y) {};
\node [input, above of=FP1, node distance = \dista] (a-FP1) {};
\node [output, below of=Signal, node distance = \distu] (u) {};
\node [input, below of=FP1, node distance = \distu, scale=0.3] (u-FP1) {};
\draw[-] (Signal) --  (y) ;
\draw[-] (y) -- node[name=u_FP1] {$\!\!\!\!\!\!\!\!\!\!\!\!\!\!\!\!\!\! y(t)$} (a-FP1);
\draw[->] (a-FP1) -- (FP1);
\draw[-] (Signal) --  (u) ;
\draw[-] (u) -- node[below, name=u_FP1] {$\!\!\!\!\!\!\!\!\!\!\!\!\!\!\!\!\!\! u(t)$} (u-FP1);
\draw[->] (u-FP1) -- (FP1);
\end{tikzpicture} 
\caption{Schematic funnel pre-compensator.}
\end{subfigure}
\hfill
\begin{subfigure}[t]{0.39\textwidth}
\includegraphics[width=\textwidth]{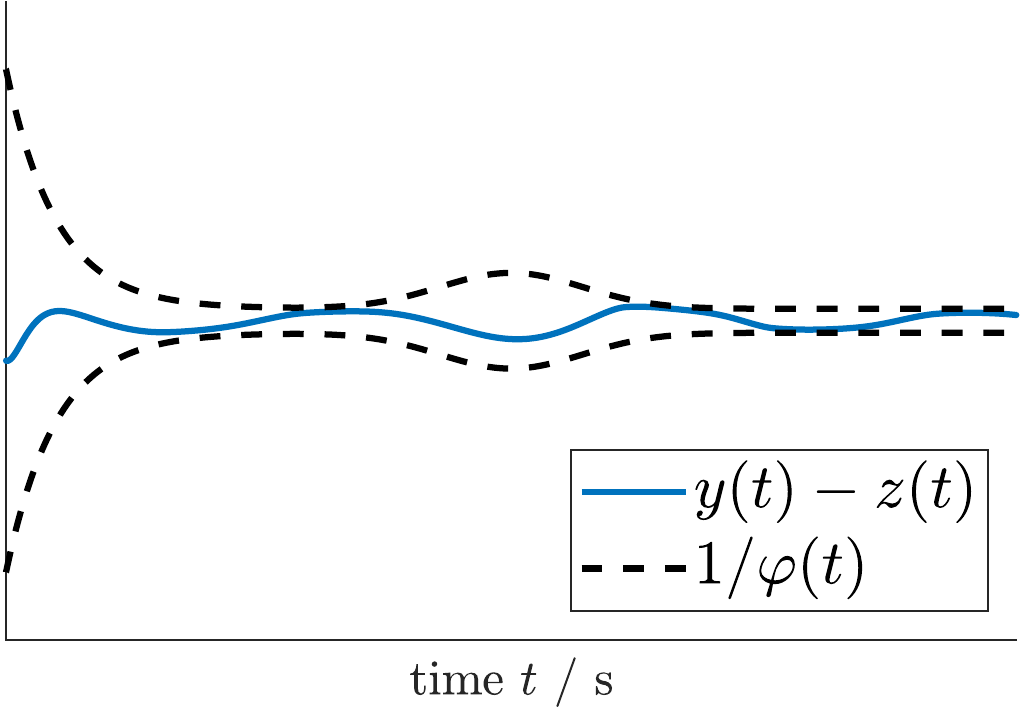}
\caption{Error~$y - z$ and funnel boundary~$1/\vp$.}
\label{Fig:FPC-Error_FIG}
\end{subfigure}
\caption{Schematic structure of an application of the funnel pre-compensator~\eqref{eq:FPC} to signals~$(u,y) \in \cP_r$. The figure is based on the respective figures in~\cite{BergReis18b}. }
 \label{Fig:FPC}
\end{figure}
The error between the signals~$y$ and~$z$, namely~$e := y-z$, evolves within a prescribed performance funnel
\begin{equation*}
\cF_\vp := \setdef{ (t,e) \in \rp \times \R^{m} }{ \vp(t) \|e(t)\| < 1}.
\end{equation*}
The situation is depicted in Figure~\ref{Fig:FPC}.
The shape of the performance funnel is determined by the funnel functions, which belong to the following set
\begin{equation*}
\Phi_r :=
\left\{
\vp \in \cC^r(\rp \to \R) \ \vline \ \begin{array}{l}
\vp, \dot \vp,\ldots,\vp^{(r)} \text{ are bounded}, \\ 
\vp(s) > 0 \text{ for all } s > 0, \\
\text{and } \liminf_{s \to \infty} \vp(s) > 0
\end{array}
\right\}.
\end{equation*}
Note that the boundary of the performance funnel is given by the reciprocal of the funnel functions, namely by~$1/\vp$. 
We highlight two important properties of the funnel functions~$\vp \in \Phi_r$. 
First, we allow~$\vp(0) = 0$ which means that the boundary has a pole at~$t=0$. This will be important in the context of funnel control, where initial conditions of the form~$\vp(0) \|e(0)\|<1$ occur, which are satisfied trivially for~$\vp(0)=0$.
Second, we do not require monotonically increasing funnel functions, see Figure~\ref{Fig:FPC-Error_FIG}. 
Although in most situations one will choose the funnel functions in such a manner that the funnel boundary is monotonically decreasing, there may occur situations where widening the funnel boundary over some time interval is beneficial, e.g., if the signal~$y$ is changing strongly or in the presence of (periodic) disturbances. 

\ \\
We recall the funnel pre-compensator $FP : \cP_r \to \cP_{r} $ proposed in~\cite{BergReis18b}, defined for $(u,\xi) \in \cP_{r}$ and $\vp \in \Phi_1$ via
\begin{equation*}
\begin{aligned}
FP({a,p},\tilde \Gamma,\vp) : (u,\xi) &\mapsto (u,\zeta_1),
\end{aligned}
\end{equation*}
where
\begin{equation} \label{eq:FPC}
\begin{aligned}
\dot \zeta_1(t) & = \big(a_{1} + p_1 h(t) \big)\big( \xi(t) - \zeta_1(t) \big) + \zeta_{2}(t) ,  &\zeta_1(0) &= \zeta_1^0 \in \R^m, \\
\dot \zeta_{2}(t) & = \big(a_{2} + p_{2} h(t) \big)\big( \xi(t) - \zeta_1(t) \big) + \zeta_{3}(t) ,   &\zeta_{2}(0) &= \zeta_{2}^0 \in \R^m, \\
& \ \, \vdots & &\\
\dot \zeta_{r-1}(t) &= \big(a_{r-1} + p_{r-1} h(t) \big)\big(\xi(t) - \zeta_1(t) \big) +  \zeta_{r}(t),  &\zeta_{r-1}(0) &= \zeta_{r-1}^0 \in \R^m,  \\
\dot \zeta_{r}(t) &= \big(a_{r} \quad + p_{r} \quad h(t) \big)\big(\xi(t) - \zeta_1(t) \big) +  \tilde \Gamma u(t),  &\zeta_{r}(0) &= \zeta_{r}^0 \in \R^m, \\
\ \\
h(t) &= \frac{1}{1 - \vp(t)^2 \| \xi(t) - \zeta_1(t) \|^2} ,
\end{aligned}
\end{equation}
and $\tilde \Gamma  \in \R^{m \times m}$, {$a := (a_1,\ldots,a_r)$, $p := (p_1,\ldots,p_r)$} and~$\vp$ are design parameters to be determined later in Section~\ref{Sec:Parameter} 

\ \\
At this stage we bring back to mind the result~\cite[Prop.~1]{BergReis18b} concerning the feasibility of the funnel pre-compensator.
It guarantees transient behaviour of the error between the signal~$y$ and the pre-compensator state~$z_1$; and the derivative~$\dot z_1$ is known exactly.
However, the higher derivatives of the pre-compensator's output, namely~$\ddot z,\ldots, z^{(r-1)}$ which explicitly depend on~$\dot y, \ldots, y^{(r-1)}$, do not approximate the higher derivatives of~$y$ in the sense that (since~$\dot y,\ldots, y^{(r-1)}$ are unknown) transient behaviour of the errors~$e_i := y^{(i-1)} - z_i$, $i=2,\ldots,r$ cannot be achieved. 
This motivates a successive application of the funnel pre-compensator, resulting in a \textit{cascade of funnel pre-compensators} as proposed in~\cite{BergReis18b}.
This means, we apply funnel pre-compensators in a row to the preceding system, which is already a funnel pre-compensator,
i.e., for~$i \in \N$ we have $FP : \cP_r \to \cP_{r}, (u,z_{i-1,1}) \mapsto (u,z_{i,1})$, the situation is depicted in Figure~\ref{Fig:Cascade}.
\begin{figure}[h!]
\begin{center}
\begin{tikzpicture}[auto, node distance=2cm,>=latex']
\def\hoch{0.8cm};
\def\breit{0cm};
\def\distu{0.8cm};
\node [block, minimum width = \breit, minimum height = \hoch,] (Signal) { $(u,y) \in \cP_r$ };
\node [block, minimum width = \breit, minimum height = \hoch, right of=Signal, node distance=2.5cm] (FP1) { $FP$ };
\node [block, minimum width = \breit, minimum height = \hoch, right of=FP1, node distance=2cm] (FP2) { $FP$ };
\node [ right of=FP2, node distance=2cm] (FP3) { $---$ };
\node [block, minimum width = \breit, minimum height = \hoch, right of=FP3, node distance=2.5cm] (FPr-1) { $FP$ };
\node [right of=FPr-1, node distance=2cm] (z) {  };
\draw[->] (Signal) to node[above]{$y(t)$} node[below]{$u(t)$} (FP1) ;
\draw[->] (FP1) to node[above]{$z_{1,1}(t)$} node[below]{$u(t)$} (FP2) ;
\draw[->] (FP2) to node[above]{$z_{2,1}(t)$} node[below]{$u(t)$} (FP3) ;
\draw[->] (FP3) to node[above]{$z_{r-2,1}(t)$} node[below]{$u(t)$} (FPr-1) ;
\draw[->] (FPr-1) -- node[above]{$z(t)$} node[below]{$u(t)$} (z) ;
\node [block, minimum width = \breit, minimum height = \hoch, above of=FP1, node distance=1.5*\distu] (Param_FP1) { ${a,p},\tilde \Gamma,\vp_1$ };
\node [block, minimum width = \breit, minimum height = \hoch, above of=FP2, node distance=1.5*\distu] (Param_FP1) { ${a,p},\tilde \Gamma,\vp$ };
\node [block, minimum width = \breit, minimum height = \hoch, above of=FPr-1, node distance=1.5*\distu] (Param_FP1) { ${a,p},\tilde \Gamma,\vp$ };
\node [output, above of=FP1, node distance = \distu] (vp1) {};
\node [output, above of=FP2, node distance = \distu] (vp_1) {};
\node [output, above of=FPr-1, node distance = \distu] (vp_r-1) {};
\draw[->] (vp1) -- (FP1);
\draw[->] (vp_1) -- (FP2);
\draw[->] (vp_r-1) -- (FPr-1);
\end{tikzpicture} 
\caption{Cascade of funnel pre-compensators~\eqref{eq:FPC-cascade} applied to signals~$(u,y) \in \cP_r$. The figure is based of the respective figure in~\cite{BergReis18b}. }
 \label{Fig:Cascade}
\end{center}
\end{figure}
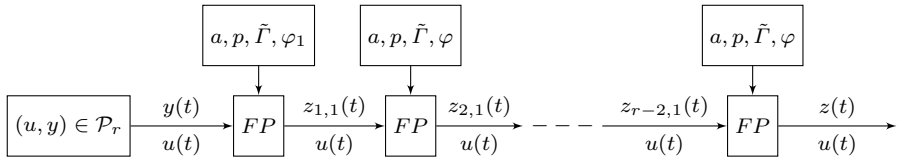
This cascade achieves an approximation~$z := z_{r-1,1}$ of the signal~$y$ with transient behaviour of the error~$y-z$, and furthermore, the higher derivatives of the funnel pre-compensator's output, namely~$\dot z,\ldots,z^{(r-1)}$ are known explicitly.
Applying the pre-compensator~$r-1$ times we obtain for~$\vp_1, \vp \in \Phi_r$
\begin{equation} \label{eq:FPC-cascade}
\begin{aligned}
FP({a,p},\tilde \Gamma,\vp) \circ \cdots \circ FP({a,p},\tilde \Gamma,\vp) \circ FP({a,p}, \tilde \Gamma,\vp_1) : \cP_r &\to \cP_{r} , \\
(u,y) & \mapsto (u, z) ,
\end{aligned}
\end{equation}
where, except of the first, all pre-compensators in the cascade have the same funnel function~$\vp$, and all have the same gain matrix~$\tilde \Gamma$;
 $\vp_1, \vp \in \Phi_r$, and ${a,p>0}$ are given by the corresponding matrices~$A,P,Q$ satisfying~\ref{Ass:A-Hurwitz}, to be introduced in Section~\ref{Sec:Parameter}.
For such a cascade of funnel pre-compensators~\cite[Thm. 1]{BergReis18b} states that it yields a system with output $z := z_{r-1,1}$ such that the error $e := y - z$ evolves within a prescribed performance funnel, and moreover, the derivatives $\dot z, \ldots, z^{(r-1)}$ are known explicitly.
An explicit representation of~$z^{(j)}$ and its dependency on the states~$z_{i,j}$ are discussed in detail in~\cite[Rem.~3]{BergReis18b}.
Figure~\ref{Fig:Dependence-on-all-states} gives a picture of the dependence of~$z^{(j)}$ on the states~$z_{i,j}$.
\begin{figure}[h!]
\begin{center}
\begin{tikzpicture}[auto, node distance=2cm,>=latex']
\def\hoch{0.0cm};
\def\breit{0cm};
\def\distu{0.65cm};
\def\dista{0.65cm};
\node [block, minimum width = \breit, minimum height = \hoch,] (zj) { $z^{(j)}$ };
\node [block, minimum width = \breit, minimum height = \hoch, right of=zj, node distance=1.8cm] (FPr-1) 
{$\begin{array}{cc}
 z_{r-1,1} & z_{r-2,1} \\
 z_{r-1,2} & \dot z_{r-2,1} \\
  \vdots & \vdots \\
  z_{r-1,j+1} & z_{r-2,1}^{(j-1)}
\end{array} $} ; 
\node [block, color=black, ,fill opacity=0.1, 
minimum width = 1.15cm, minimum height = 1.45cm, 
below of = FPr-1, node distance = 0.18cm, xshift=0.75cm] (box-FPr-1) {} ;
\node [block, minimum width = \breit, minimum height = \hoch, right of=box-FPr-1, node distance=1.92cm] (FPr-2) 
{$\begin{array}{cc}
 z_{r-2,1} & z_{r-3,1} \\
 z_{r-2,2} & \dot z_{r-3,1} \\
  \vdots & \vdots \\
  z_{r-2,j} & z_{r-3,1}^{(j-2)}
\end{array} $} ; 
\node [block, color=black, ,fill opacity=0.1, 
minimum width = 1.15cm, minimum height = 1.45cm, 
below of = FPr-2, node distance = 0.18cm, xshift=0.6cm] (box-FPr-2) {} ;
\node[right of = box-FPr-2, node distance=1.1cm] (between) {$\cdots$} ;
\node [block, minimum width = \breit, minimum height = \hoch, right of=between, node distance=1.73cm] (FPr-j) 
{$\begin{array}{cc}
 z_{r-j+1,1} & z_{r-j,1} \\
 z_{r-j+1,2} & \dot z_{r-j,1} \\
  z_{r-j+1,3} & 
\end{array} $} ; 
\node [block, color=black, ,fill opacity=0.1, 
minimum width = 1.1cm, minimum height = 0.7cm, 
below of = FPr-j, node distance = 0.2cm, xshift=0.75cm] (box-FPr-j) {} ;
\node [block, minimum width = \breit, minimum height = \hoch, right of=box-FPr-j, node distance=2cm] (FPr-j-1) 
{$\begin{array}{cc}
 z_{r-j,1} & z_{r-j-1,1} \\
 z_{r-j,2} & 
\end{array} $} ;
\draw[->] (FPr-1) -- (zj);
\draw[->] (FPr-2) -- (box-FPr-1);
\draw[->] (between) -- (box-FPr-2);
\draw[->] (FPr-j) -- (between);
\draw[->] (FPr-j-1) -- (box-FPr-j);
\end{tikzpicture} 
\caption{Dependence of the derivatives~$z^{(j)}$ on the intermediate pre-compensator states. The figure is based on the respective figure in~\cite{BergReis18b}.}
 \label{Fig:Dependence-on-all-states}
\end{center}
\end{figure}
At the first glance, the expression for~$z^{(j)}$ in~\cite[Rem.~3]{BergReis18b} looks lengthy and awkward to handle. 
However, we highlight that with the aid of the given formula in~\cite[Rem.~3]{BergReis18b} the computation of all required derivatives of~$z$ can be performed completely algorithmically.

\section{Main result: application of the funnel pre-compensator to minimum phase systems} \label{Sec:main-result-FPC-min-phase}
In this section we face the open question formulated in~\cite[Rem. 4]{BergReis18b}, namely if the interconnection of a minimum phase system with a cascade of funnel pre-compensators yields a minimum phase system for relative degree larger than three.
A careful inspection reveals that the proof of~\cite[Thm. 2]{BergReis18b} is incomplete as regards the boundedness of~$h_1, h_2$ in the case~$r=3$; this is, however, resolved in the present article.
We show that for arbitrary~$r \in \N$ the interconnection of a cascade of $r-1$~funnel pre-compensators with a minimum phase system with relative degree~$r$ yields a system of the same relative degree which is minimum phase as well, and the first~$r-1$ derivatives of the interconnection's output, this is, the funnel pre-compensator's output~$z$, are known explicitly.

\ \\
\subsection{System class} \label{Sec:System-class}
We recall the system class under investigation in~\cite{BergReis18b}.
First, 
we introduce the following class of operators. 
%
\begin{Definition} \label{Def:OP-T} 
If for~$\sigma > 0$ and $n,q \in \N$ the operator $T: \cC([-\sigma,\infty) \to \R^{n}) \to \cL_{\rm loc}^\infty(\rp \to \R^q)$ has the following properties
\begin{enumerate}[label = (\alph{enumi}), ref=(\alph{enumi}),leftmargin=*]
\item \label{T:BIBO}
 $T$ maps bounded trajectories to bounded trajectories, i.e., for all $c_1 > 0$, there exists $c_2>0$ such that for all 
$\xi \in \cC([-\sigma,\infty) \to \R^n)$,
\[
\sup_{t \in [-\sigma,\infty)} \| \xi(t) \| \le c_1 \ \Rightarrow \ \sup_{t \in [0,\infty)} \| T(\xi)(t) \| \le c_2,
\]
\item \label{T:causal}
$T$ is causal, i.e., for all $t \ge 0$ and all $\zeta, \xi \in \cC([-\sigma,\infty) \to \R^n)$, 
\[
\zeta |_{[-\sigma,t)} = \xi|_{[-\sigma,t)} \ \Rightarrow \ T(\zeta)|_{[0,t)} \overset{a.a.}{=} T(\xi)|_{[0,t)},
\]
\item \label{T:Lipschitz}
$T$ is locally Lipschitz continuous in the following sense: for all $t \ge 0 $ and all $\xi \in \cC([-\sigma,t] \to \R^n)$ 
there exist $\Delta, \delta, c > 0$ such that for all 
$\zeta_1, \zeta_2 \in \cC([-\sigma,\infty) \to \R^n)$ with $\zeta_1|_{[-\sigma,t]} = \xi $, $\zeta_2|_{[-\sigma,t]} = \xi $ 
and $\|\zeta_1(s) - \xi(t)\| < \delta$,  $\|\zeta_2(s) - \xi(t)\| < \delta$ for all $s \in [t,t+\Delta]$ we have
\[
\esssup_{s \in [t,t+\Delta]}\| T(\zeta_1)(s) - T(\zeta_2)(s) \| \le c \ \rm{sup}_{s \in [t,t+\Delta]} \| \zeta_1(s) - \zeta_2(s)\| ,
\]
\end{enumerate}
then we say the operator~$T$ belongs to the operator class~$\cT_\sigma^{n,q}$.
\end{Definition}
With this, we introduce the system class~$\cN^{m,r}$ which is the same class of systems under consideration in~\cite{BergReis18b}, namely (multi-input multi-output) systems with stable internal dynamics, and the system's input and output have the same dimension. 
\begin{Definition}
For a system
\begin{equation} \label{eq:System}
\begin{aligned}
y^{(r)}(t) &= \sum_{i=1}^r R_i y^{(i-1)}(t) + f\left(d(t), T\big(y,\dot y,\ldots,y^{(r-1)}\big)(t) \right) + \Gamma \, u(t), \\
y|_{[-\tau,0]} &= y^0 \in \cW^{r-1, \infty}([-\tau,0] \to \R^m),
\end{aligned}
\end{equation}
where $\tau > 0$ is the ``memory" of the system, i.e., an initial trajectory is given, $r \in \N$ is the relative degree, and
for~$p \in \N$ the ``disturbance" satisfies $d \in \cL^\infty(\rp \to \R^p)$,
for~$q \in \N$ we have $f \in \cC(\R^p \times \R^q \to \R^m)$,
the high gain matrix~$\Gamma$ is symmetric and sign definite (w.l.o.g. we assume $0<\Gamma = \Gamma^\top \in \R^{m \times m}$),
and the operator~$T$ belongs to the class~$\cT^{rm,q}_\tau$, 
we say system~\eqref{eq:System} belongs to the class~$\cN^{m,r}$, and we write 
\begin{equation*}
(d,f,T,\Gamma) \in \cN^{m,r}.
\end{equation*}
\end{Definition}
The function $u:\rp \to \R^m$ is called \textit{input}, the function $y : \rp \to \R^m$ \textit{output} of system~\eqref{eq:System}, respectively.
Note that the input and the output have the same dimension.
Condition~\ref{T:BIBO} in Definition~\ref{Def:OP-T} resembles a minimum phase property, more precise, an input to state stability of the internal dynamics of system~\eqref{eq:System}, 
where from the viewpoint of the internal dynamics the system's output and its derivatives act as inputs.

\ \\
For fixed input $u \in \cL^\infty(\rp \to \R^m)$ a function $y \in \cC^{r-1}([-\tau,\omega) \to \R^m)$ is called 
\textit{solution} of~\eqref{eq:System} on an interval~$[-\tau,\omega)$, where~$\omega \in (0,\infty]$, 
if $y|_{[-\tau,0]} = y^0$ and $y^{(r-1)}|_{[0,\omega)}$ is weakly differentiable and satisfies~\eqref{eq:System} for almost all~$t \in [0,\omega)$.
A solution~$y$ is called \textit{maximal solution}, if it has no right extension that is also a solution. 

\begin{Remark} \label{Rem:linear-systems}
An important subclass of~\eqref{eq:System} are linear systems of the form 
\begin{equation} \label{eq:linear-system}
\begin{aligned}
\dot x(t) &=  Ax(t) +   B u(t) + d(t) , & x(0) &= x^0 \in \R^n,\\
y(t) &=  Cx(t),
\end{aligned}
\end{equation}
where~$u$ denotes the input, and~$y$ the output of the system, respectively; further we have the system matrix $ A \in \R^{n \times n}$, 
the input distribution matrix~$ B \in \R^{n \times m}$ and the linear output measurement $ C : \R^{n} \to \R^m$, this is, $ C \in \R^{m \times n}$ for $m \le n$, and~$\rk C = \rk B = m$; note that the dimensions of the input and the output are equal.
Let
\begin{equation*}
\begin{aligned}
\cD(\rp \to \R^n) := \setdef{ d \in \cL^\infty(\rp \to \R^n)}{ \begin{array}{l}
\forall \, j=0,\ldots,r-1: \\
CA^{j} d \in \cW^{r-1-j,\infty}(\rp \to \R^n) 
\end{array}  } .
\end{aligned}
\end{equation*}
Note that~$\cW^{r-1,\infty}(\rp \to \R^n) \subset \cD(\rp \to \R^n)$.
Now, if
\begin{equation} \label{def:rel-dg-linear-system}
\begin{aligned}
d & \in \cD(\rp \to \R^n), \\
 \forall\, k  &\in \{0,\ldots,r-2\} :  C  A^{k}   B  = 0 \ \text{and} \ \Gamma :=  C  A^{r-1}  B \in \Gl_m(\R),
\end{aligned}
\end{equation}
then - straightly following the derivations and calculations in~\cite[Thm.~3]{IlchWirt13} - with
\begin{equation*}
\begin{aligned}
& \cB := \begin{bmatrix}
B & AB & \ldots & A^{r-1} B
\end{bmatrix} \in \R^{n \times rm} , \
 \cC := \begin{bmatrix}
C^\top & (CA)^\top & \ldots & (CA^{r-1})^\top
\end{bmatrix}^\top \in \R^{rm \times n} , \\
%
%
& V \in \R^{n \times (n-rm)} \text{ s.t. } \im V = \ker \cC , \ N := V^\dagger ( I_n - \cB (\cC \cB)^{-1} \cC ) \in \R^{(n-rm) \times n}, \\
& U := \begin{bmatrix}
\cC \\ N
\end{bmatrix} \in \Gl_n(\R),
\end{aligned}
\end{equation*}
and the operator
\begin{equation*}
\begin{aligned}
L  : \cD(\rp \to \R^n) &\to \cL^\infty(\rp \to \R^{n}) , \\
d(\cdot) &\mapsto \left( t \mapsto \begin{pmatrix}
l_1(t) \\
\vdots \\
l_{r}(t) \\
0_{n-rm}
\end{pmatrix}
\right) , \ l_i(t) = \sum_{j=0}^{i-2} CA^{j} d^{(i-2-j)}(t), \ i= 1,\ldots,r
\end{aligned}
\end{equation*}
the change of coordinates
\begin{equation*}
\begin{pmatrix}
\xi_1 \\
\vdots \\
\xi_r \\
\eta
\end{pmatrix} 
= U x + L(d) =
\begin{pmatrix}
y \\
\vdots \\
y^{(r-1)} \\
\eta
\end{pmatrix} 
\end{equation*}
transforms system~\eqref{eq:linear-system} into \textit{Byrnes-Isidori form}
\begin{subequations} \label{eq:BIF-linear}
\begin{equation} \label{eq:BIF-linear-system}
\begin{aligned}
\dot \xi_i(t) &= \xi_{i+1}(t),  &\xi_{i}(0) &= \xi_{i}^0 \in \R^m, \\
\dot \xi_r(t) &= \sum_{j=1}^{r} R_j \xi_j(t) + S \eta(t) + \Gamma u(t) + d_r(t) , &\xi_r(0) &= \xi_r^0 \in \R^m, \\
\dot \eta(t) &= Q \eta(t) + P \xi_1(t) + d_\eta(t)  , & \eta(0) &= \eta^0 \in \R^{n-rm},
\end{aligned}
\end{equation}
with output
\begin{equation} \label{eq:BIF-linear-output}
y(t) = \xi_1(t),
\end{equation}
\end{subequations}
where
\begin{equation*}
\begin{aligned}
\begin{bmatrix}
R_1 & \ldots & R_r & S
\end{bmatrix} & := CA^r U^{-1} \in \R^{m \times (rm + (n-rm))}, \\
P & := NA^r B \Gamma^{-1} \in \R^{(n-rm) \times m}, \ Q = NAV \in \R^{(n-rm) \times (n-rm)}, \\
d_r(t) & := \sum_{j=0}^{r-1} \left( CA^{j} d^{(r-1-j)}(t) - R_{j+1} l_{j+1}(t) \right) \in \cL^\infty(\rp \to R^m) , \\
d_\eta(t) & := N \left( d(t) - AU^{-1} T(d)(t) \right) \in \cL^\infty(\rp \to \R^{n-rm}) ,
\end{aligned}
\end{equation*}
and the high-gain matrix~$\Gamma$ is given in~\eqref{def:rel-dg-linear-system}.
The last differential equation in~\eqref{eq:BIF-linear-system} describes the internal dynamics of system~\eqref{eq:linear-system}.
We associate the (linear) integral operator
\begin{equation} \label{def:integral-operator}
J: y(\cdot) \mapsto \left(t \mapsto \int_{0}^{t} e^{Q(t-s)} P y(s) \, \ds  \right)
\end{equation}
with the internal dynamics in~\eqref{eq:BIF-linear-system} and obtain for 
${H(\cdot) := e^{Q \cdot} [0, I_{n-rm}] U x^0}$ 
and $D(t) := e^{Qt} \left(d_\eta(0) + \int_0^{t} e^{-Qs} d_\eta(s) \, \ds \right)$
the internal state
\begin{equation*}
\eta(t) =  D(t) + H(t) + J(y)(t). 
\end{equation*}
With this and~\eqref{eq:BIF-linear} we find that~\eqref{eq:linear-system} is equivalent to the functional differential equation
\begin{equation*}
y^{(r)}(t) = \sum_{i=1}^{r} R_i y^{(i-1)}(t) + f\big(S(D(t)+H(t)),SJ(y)(t) \big)+ \Gamma u(t) + d_r(t),
\end{equation*} 
where $f(v,w) = v + w$ for $v,w \in \R^m$, and the operator~$J$ satisfies conditions~\ref{T:causal},\ref{T:Lipschitz} of Definition~\ref{Def:OP-T}.
The minimum phase property (condition~\ref{T:BIBO}) for linear systems and its various equivalent conditions have been studied extensively, see e.g. 
\cite{ByrnWill84,Berg14a,TrenStoo01}.
Here, we restrict ourself to mention the equivalence between system~\eqref{eq:linear-system} being minimum phase, i.e., $\sigma(Q) \subseteq \C_-$, and having asymptotically stable zero dynamics (see e.g.~\cite{IlchWirt13}), where the latter means (cf.~\cite{IlchRyan07,Isid95})
\begin{equation} \label{eq:minimum-phase}
 \forall\,  \lambda \in \C_- \, : \  \rk \begin{bmatrix}  A - \lambda I  &  B \\   C & 0  \end{bmatrix} = n + m.
 \end{equation} 
Note that for~$\sigma(Q) \subseteq \C_-$ we have~$D \in \cL^\infty(\rp \to \R^m)$.
Therefore, if system~\eqref{eq:linear-system} 
has relative degree~$r \in \N$ as in~\eqref{def:rel-dg-linear-system} and satisfies~\eqref{eq:minimum-phase} it is contained in the system class~$\cN^{m,r}$.
Note that if the commonly used assumption is satisfied that the disturbance does not affect the integrator chain but enters the system on the input's level (cf.~\cite{Berg20,Berg21}), i.e.,
\begin{equation*}
\forall\, k=0,\ldots,r-2 \, : \ CA^kd(\cdot) = 0,
\end{equation*}
then~$\cD(\rp \to \R^n) = \cL^\infty(\rp \to \R^n)$; 
in this case we have~$(\xi^\top , \eta^\top)^\top = Ux$ and~$(d_r^\top, d_\eta^\top)^\top = [(CA^{r-1})^\top, N^\top]^\top d$.
\end{Remark}

\ \\
We conclude this subsection with the preceding remark and turn towards the funnel pre-compensator's design parameters.

\subsection{The pre-compensator's design parameters} \label{Sec:Parameter}
We introduce the set of feasible design parameters for the funnel pre-compensator. 
For~$a = (a_1,\ldots,a_{r})^\top \in \R^r$, $p = (p_1,\ldots,p_r)^\top \in \R^r$ we set
\begin{equation*} 
\Sigma := \setdef{ \begin{array}{l}
 a, p \in \R^r, \\
\vp, \vp_1 \in \Phi_r, \\
\rho \in \R, \\
\tilde \Gamma \in \R^{m \times m} \\
\end{array} }{\text{\ref{Ass:A-Hurwitz}~--~\ref{Ass:G} hold}} ,
\end{equation*}
where~\ref{Ass:A-Hurwitz}~--~\ref{Ass:G} denote the following properties.
\begin{enumerate}[label = \textbf{(A.\arabic{enumi})}, ref=(A.\arabic{enumi}),leftmargin=*]
\item \label{Ass:A-Hurwitz}
The numbers $a_{i}$ are such that~$a_i > 0$ for all $i=1,\ldots,r$, and
\begin{align*} 
A := \begin{bmatrix}
-a_1 & 1 & &\\ \vdots & & \ddots &  \\ -a_{r-1} & & & 1 \\ -a_{r} & & &0
\end{bmatrix} \in \R^{r \times r}
\end{align*}
is Hurwitz, i.e. $\sigma(A) \subseteq \C_-$. 
Furthermore, let
$P = \begin{smallbmatrix}  P_1 &  P_{2} \\  P_{2}^\top &  P_{4} \end{smallbmatrix} > 0$, with $ P_1 \in \R$, $ P_{2} \in \R^{1 \times (r-1)}$, $ P_{4} \in \R^{(r-1) \times (r-1)}$ 
be the solution of
\begin{align*} 
A^\top P + P A + Q = 0 
\end{align*}
for some $Q \in \R^{r \times r}$ with $Q = Q^\top > 0$; then~$p$ is defined as
\begin{equation*} 
\begin{pmatrix}
p_{1 }\\ \vdots \\ p_{r} \end{pmatrix} :=
P^{-1}\begin{pmatrix}
 P_1 -  P_{2}  P_{4}^{-1}  P_{2}^\top  \\ 0 \\ \vdots \\ 0
\end{pmatrix} 
= \begin{pmatrix}
1 \\ -  P_{4}^{-1}  P_{2}^\top 
\end{pmatrix}.
\end{equation*}
%
\item \label{Ass:funnel-functions}
The funnel functions~$\vp_1, \vp \in \Phi_r$ from~\eqref{eq:FPC-cascade} satisfy
\begin{equation*}
\exists \, \rho >1 \ \forall\, t\ge0: \  \vp(t) = \rho \, \vp_1(t).
\end{equation*}
%
\item \label{Ass:Gam-tildeGam}
The matrix $\tilde \Gamma$ from~\eqref{eq:FPC-cascade} is symmetric and sign definite (w.l.o.g. we assume~$\tilde \Gamma > 0$) and moreover, for $\Gamma = \Gamma^\top > 0$ from~\eqref{eq:System} we have 
\begin{equation*}
{\Gamma \tilde \Gamma^{-1} = \left(\Gamma \tilde \Gamma^{-1} \right)^\top > 0}. 
\end{equation*}
%
\item  \label{Ass:G}
For~$\Gamma$ from~\eqref{eq:System},~$\tilde \Gamma$ from~\eqref{eq:FPC-cascade} and~$\rho$ from~\ref{Ass:funnel-functions} the matrix $G := I_m - \Gamma \tilde \Gamma^{-1}$ satisfies
\begin{equation*} 
\|G\| < \min \left\{\frac{\rho - 1}{r-2} , \, \frac{\rho}{4\rho^2 (\rho+1)^{r-2} -1} \right\} .
\end{equation*}
\end{enumerate}
If conditions~\ref{Ass:A-Hurwitz}~--~\ref{Ass:G} hold we write~$(a,p,\vp,\vp_1,\rho,\tilde \Gamma) \in \Sigma$.
Condition~\ref{Ass:funnel-functions} means that the first funnel which limits the error~$y-z_{1,1}$ is somewhat tighter than the others; property~\ref{Ass:Gam-tildeGam} in particular asks for regularity of the matrix product~$\Gamma \tilde \Gamma^{-1}$, and~\ref{Ass:G} ensures that the matrix~$\tilde \Gamma$ is ``not too different" from matrix~$\Gamma$.
\begin{Remark} \label{Rem:Constants}
At the first glance, there are a lot of parameters to be chosen appropriately satisfying~\ref{Ass:A-Hurwitz}.
However, consider the polynomial $(s+s_0)^r$, which has all its roots in~$\C_-$ for~$s_0 > 0$. Then, with
\begin{equation*}
(s+s_0)^r = s^r + \sum_{i=1}^{r} a_i s^{r-i}
\end{equation*}
we obtain
\begin{equation*}
A := \begin{bmatrix}
-a_1 & 1 & &\\ \vdots & & \ddots &  \\ -a_{r-1} & & & 1 \\ -a_{r} & & &0
\end{bmatrix} \in \R^{r \times r}, \quad \sigma(A) \subseteq \C_-, 
\end{equation*}
this is, the matrix~$A$ is Hurwitz.
Moreover, the simple choice~$Q = I_m$ is always feasible; and the matrix~$P$ is completely determined by the choice of~$A$ and~$Q$.
Therefore, since the constants $p_1,\ldots,p_r$ are given via the Lyapunov matrix~$P$, all parameters required to satisfy~\ref{Ass:A-Hurwitz} 
can be determined by choosing the real number~${s_0 > 0}$.
\end{Remark}

\subsection{The funnel pre-compensator applied to minimum phase systems} \label{Sec:FPC}
In this section we show that the conjunction of a system~\eqref{eq:System} with a cascade of funnel pre-compensators as in~\eqref{eq:FPC-cascade} leads to a minimum phase system. 
To this end, we prove the extension of~\cite[Thm. 2]{BergReis18b} for arbitrary relative degree~$r \in \N$.

\ \\
We show that the minimum phase property of system~\eqref{eq:System}, modelled by property~\ref{T:BIBO} of Definition~\ref{Def:OP-T} 
 is preserved by the 
conjunction of system~\eqref{eq:System} with the cascade of funnel pre-compensators~\eqref{eq:FPC-cascade}.
To this end, we require that the operator~$T$ satisfies a stronger condition.
\begin{Definition}
For~$r,m \in \N$, $n = rm$ and~$1\le k \le r$ we define the operator 
class~$\cT^{n,q}_{\sigma,k} := \setdef{ T \in \cT^{n,q}_\sigma}{ T~\text{satisfies~\ref{cond:a_k}}} \subseteq \cT^{n,q}_\sigma$ (equality if~$k=r$), where
\begin{enumerate}[label = ($\alph{enumi}'_k$), ref=($\alph{enumi}'_k$),leftmargin=*]
\item \label{cond:a_k}
for all $c_1 > 0$ there exists $c_2 > 0$ such that for all $\xi_1,\ldots,\xi_r \in \cC([-\sigma,\infty) \to\R^m)$  
\begin{equation*}
\sup_{t \in [-\tau,\infty)} \left\| \left(\xi_1(t)^\top,\ldots,\xi_{k}(t)^\top \right) \right\|  \le c_1 \ \Rightarrow \ \sup_{t \in [0,\infty)} \| T(\xi_1,\ldots,\xi_r)(t) \| \le c_2. 
\end{equation*}
\end{enumerate}
\end{Definition}
Now, the stronger condition on the operator~$T$, namely boundedness whenever the first component of its input is bounded, reads~$T \in \cT^{{rm},q}_{\sigma,1}$, this is, $T$ is bounded whenever~$y: \rp \to \R^{m}$ is bounded.
Furthermore, we set \[ \cN^{m,r}_k := \setdef{ (d,f,T,\Gamma) \in \cN^{m,r} }{ T \in \cT^{n,q}_{\sigma,k} } \subseteq \cN^{m,r} \text{ (equality if~$k=r$).} \]
\begin{Remark}
The somewhat arcane condition~$T \in \cT^{rm,q}_{\sigma,1}$ reflects the intuition that in order to have the conjunction of the system with the funnel pre-compensator being a minimum phase system, we must be able to conclude from the available information (only the output~$y$) that the internal dynamics stay bounded. 
Note that this, however, \textit{does not} mean that the operator~$T$ does not act on the output signal's derivatives but on~$y$ only, see Example~\ref{Ex:tracking}.
\end{Remark}

\begin{Remark} \label{Rem:linear-min-phase-satisfy-a-prime}
We highlight that, if for~$Q$ in~\eqref{eq:BIF-linear-system} we have~$\sigma(Q) \subseteq \C_-$, this is, if condition~\eqref{eq:minimum-phase} is satisfied, the operator defined in~\eqref{def:integral-operator} is contained in~$\cT^{rm,q}_{\sigma,1}$.
This means, the class of linear minimum phase systems which satisfy~\eqref{def:rel-dg-linear-system} is encompassed by the system class~$ \cN^{m,r}_1$.
Moreover, $\cN^{m,r}_1$ encompasses systems of the following form
\begin{equation} \label{eq:System-within-class}
\begin{aligned}
y(t) &= \xi_1(t), \\
\dot \xi_i(t) &= \xi_{i+1}(t), \quad i=1,\ldots,r-1, \\
\dot \xi_r(t) &= f(t,\xi(t),\eta(t)) + \Gamma u(t), \\
\dot \eta(t) &= g(\eta(t),\xi_1(t)), \\
\end{aligned}
\end{equation}
where for~$rm \le n \in \N$ the function~$f: \rp \times \R^{rm} \times \R^{n-rm} \to \R^{m}$ is locally Lipschitz in~$(\xi,\eta) \in \R^{rm} \times \R^{n-rm}$, and piecewise continuous and bounded in~$t$; 
${g: \R^{n-rm} \times \R^{m} \to \R^{n-rm}}$ is such that for~$\xi_1 \in \cL^{\infty}(\rp \to \R^m)$ the corresponding ODE has a bounded solution (see e.g.~\cite[Thm. 4.3]{Lanz21}); and~$\Gamma \in \Gl_m(\R)$ is symmetric and sign definite.
Therefore, with constant input parameter a subclass of the class of systems under consideration in~\cite{ChowKhal19} is contained in~$ \cN^{1,r}_1 \subset  \cN^{m,r}_1$.
Moreover, the class~$\cN^{m,r}_1$ encompasses the system class under consideration in~\cite{IlchRyan07}.
In~\cite[Cor.~5.7]{ByrnIsid91a} explicit criteria on the parameters~$a,b$ of nonlinear systems of the form $\dot x(t) = a(x(t)) + b(x(t)) u(t)$ are given such that it can be transformed into a system~\eqref{eq:System-within-class}.
\end{Remark}

\ \\
We present a version of~\cite[Thm. 2]{BergReis18b} (in particular without the restriction~${r \in \{ 2,3 \}}$), and hereinafter prove it.
\begin{Theorem} \label{Thm:FPC-min-phase}
Consider a system~\eqref{eq:System} with 
$(d,f,T,\Gamma) \in  \cN^{m,r}_1$ (note that $T \in \cT^{rm,q}_{\tau,1}$, $q \in \N$) and $y^0 \in \cW^{r-1,\infty}([-\tau,0] \to \R^m)$. 
Further consider the cascade of funnel pre-compensators defined by~\eqref{eq:FPC-cascade} with~$(a,p, \vp,\vp_1,\rho, \tilde \Gamma ) \in \Sigma $
and assume the initial conditions
\begin{equation} \label{eq:inital}
\vp_1(0)\| y(0) - z_{1,1}^0 \| < 1, \quad \vp(0)\| z_{i-1,1}^0 - z_{i,1}^0 \| < 1, \quad i=2,\ldots,r-1,
\end{equation}
are satisfied.
Then, for $\bar q = rm(r-1)+r$, there exist $\tilde d \in \cL^\infty( \rp \to \R^r$), $\tilde F \in \cC( \R^r \times \R^{\bar q} \to \R^m)$ 
and an operator $\tilde T : \cC([-\tau,\infty) \to \R^{rm}) \to \cL_{\rm loc}^\infty(\rp \to \R^{\bar q})$ with
\begin{equation*}
(\tilde d,\tilde F,\tilde T, \tilde \Gamma) \in \cN^{m,r} 
\end{equation*}
such that
the conjunction of~\eqref{eq:FPC-cascade} and~\eqref{eq:System} with input~$u$ and output~$z : = z_{r-1,1}$ can be equivalently written as
\begin{equation} \label{eq:ddt_r-z}
z^{(r)}(t) = \tilde F\big( \tilde d(t), \tilde T(z,\dot z,\ldots,z^{(r-1)})(t) \big) + \tilde \Gamma u(t), 
\end{equation}
with respective initial conditions.
\end{Theorem}

\ \\
Since the proof is quite long and partly technical we present a sketch of it here; the proof itself is relegated to the Appendix and is subdivided in three main steps.
In the first step we recall the transformations given in~\cite[pp.~4759-4760]{BergReis18b} which allow to analyse the error dynamics of two successive pre-compensators. 
The second step is the main part of the proof consisting of preparatory work to show that there exists an operator~$\tilde T \in \cT^{rm,\bar q}_\tau$ 
such that the conjunction of a minimum phase system~\eqref{eq:System} with a cascade of funnel pre-compensators~\eqref{eq:FPC-cascade} can be written as in~\eqref{eq:ddt_r-z}; 
the functions~$\tilde d$ and~$\tilde F$ are then given naturally.
We define an operator~$\tilde T$ mapping the pre-compensator's output~$z$ and its derivatives to the state of the overall auxiliary error-system~\eqref{eq:ddt_w} and the respective gain functions. 
In order to show that~$\tilde T$ satisfies condition~\ref{T:BIBO} in Definition~\ref{Def:OP-T} we establish boundedness of the solution of the auxiliary system~\eqref{eq:ddt_w} and the respective gain functions.
Here, step two splits into two parts.
First, using~$T \in \cT^{rm,q}_{\tau,1}$ - more specifically we use that $T(y,\ldots,y^{(r-1)})$ is bounded whenever~$y$ is bounded - we may define an overall system of errors of two successive pre-compensators, namely sytem~\eqref{eq:ddt_w-overall} for~$i=3,\ldots,r-1$
\begin{equation*} 
\begin{aligned}
\dot w_1(t) &= \hat A w_1(t) - h_1(t) \bar P \Gamma \tilde \Gamma^{-1} \bar w(t) + B_1(t) , \\
\dot w_2(t) &= \hat A w_2(t) - h_2(t) \bar P w_{2,1}(t) + h_1(t) \bar P \bar w(t) + B_2(t)  , \\
\dot w_i(t) &= \hat A w_i(t) - h_i(t) \bar P w_{i,1}(t) + h_{i-1}(t) \bar P w_{i-1,1}(t) + B_i(t) ,
\end{aligned}
\end{equation*}
where the error-states~$w_{i,j}$ stem from the transformations in Step~1 and the compact states~$w_i, \bar w$ are defined at the beginning of Step~2a; $B_1, B_2, B_i$ are bounded functions. 
For this overall error-system~\eqref{eq:ddt_w-overall} we find a Lyapunov function, which in combination with Gr\"onwall's lemma allows us to deduce boundedness of the error-states~$w_i$.
In the second part of step two, which is the most technical part of the proof, we show that the gain functions~$h_i$ are bounded. 
This demands particular accuracy since the functions~$h_i$ may introduce singularities.
Due to the shape of the gain functions, namely~$h(t) = \left( 1 - \vp(t)^2\|x(t)\|^2 \right)^{-1}$ boundedness is equivalent to the existence of~$\nu > 0$ 
such that $\|x(t)\| \le \vp(t)^{-1} - \nu$, which is commonly utilized in the standard funnel control proofs, cf.~\cite[pp.~350-351]{BergLe18a}.
However, unlike the standard funnel case, the dynamics under consideration involve the previous gain function, respectively, and the first equation involves the last gain function; to see this, we recall~\eqref{eq:ddt-w-bounded} for~$i=2,\ldots,r-1$
\begin{equation*} 
\begin{aligned}
\ddt \tfrac{1}{2} \| x_1(t)\|^2 &= -h_1(t) \| x_1(t)\|^2 + h_{r-1}(t) x_1(t)^\top G x_{r-1}(t) + x_1(t)^\top b_1(t), \\
\ddt \tfrac{1}{2} \| x_{i}(t)\|^2 &= -h_i(t) \| x_{i}(t)\|^2 + h_{i-1}(t) x_{i}(t)^\top x_{i-1}(t) + x_{i}(t)^\top b_i(t),
\end{aligned}
\end{equation*}
where~$x_i$ are auxiliary states defined in Step~2b; $b_1,b_i$ are bounded functions.
It turns out that this loop structure demands some technical derivations and requires accurate estimations of the involved expressions.
Exploiting properties~\ref{Ass:A-Hurwitz}~--~\ref{Ass:G} of the design parameters we can show by contradiction that there exist~$\kappa_i > 0$ such that $\|x_i(t)\| \le \vp(t)^{-1} - \kappa_i$ for all~$i=2,\ldots,r-1$ (and $\|x(t)_1\| \le \vp_1(t)^{-1} - \kappa_1$), respectively, which implies boundedness of all gain functions~$h_i$.
In step three we summarize the previously established results to deduce~$\tilde T \in \cT^{rm,\bar q}_{\tau}$.
Then, the functions~$\tilde d$ and~$\tilde F$ arise naturally in equation~\eqref{Def:tildeF}.
Together, we may conclude that the conjunction of a minimum phase system~\eqref{eq:System}, where $(d,f,T,\Gamma) \in \cN^{m,r}_1$, 
with a cascade of funnel pre-compensators~\eqref{eq:FPC-cascade}, where~$(a,p,\vp,\vp_1,\rho, \tilde \Gamma) \in \Sigma$, can be equivalently written as a minimum phase system~\eqref{eq:ddt_r-z} with~$(\tilde d, \tilde F, \tilde T, \tilde \Gamma) \in \cN^{m,r}$.

\ \\
As a direct consequence of the pre-compensator's design, namely to be of funnel type, we have the following result.
\begin{Corollary}
We use the notation and assumptions from Theorem~\ref{Thm:FPC-min-phase}. 
Then, for any $u \in \cL_{\rm loc}^\infty(\rp \to \R^m)$ and any solution of~\eqref{eq:FPC-cascade},\eqref{eq:System} with initial conditions~\eqref{eq:inital} we have
\begin{equation} \label{eq:transient-y-z}
\exists\, \ve>0 \ \forall \ t>0: \ \|y(t) - z(t) \| <  (\rho + r-2) \vp(t)^{-1} - \ve.
\end{equation}
\end{Corollary}
\begin{proof}
The prescribed transient behaviour~\eqref{eq:transient-y-z} follows directly from an iterative application of~\cite[Prop.~1]{BergReis18b}.
\qed
\end{proof}

\begin{Remark}
A careful inspection of the proof of Theorem~\ref{Thm:FPC-min-phase} reveals that conditions~\ref{Ass:funnel-functions}~--~\ref{Ass:G} on the design parameters
are sufficient but far from necessary.
Condition~\ref{Ass:G} on the norm of the matrix~${G = I_m -  \Gamma \tilde \Gamma^{-1}}$ can be interpreted as a ``small gain condition" 
as conjectured in~\cite[Rem. 4]{BergReis18b}; roughly speaking it means ``choose the matrix~$\tilde \Gamma$ close enough to the matrix~$\Gamma$".
Examining the proof shows that this condition plays a crucial role in the estimations~\eqref{eq:estimate_annoying-term} and~\eqref{eq:h1-h_r-1}, however, it allows for various reformulations and small changes which still are sufficient to prove the theorem; especially, the condition~$\|G\| < \rho/(4\rho^2(\rho+1)^{r-2} -1)$ has many varieties - we cannot claim having found the weakest.
If~$\Gamma$ is known, the simple choice $\tilde  \Gamma = \Gamma$ is feasible and the proof simplifies significantly; moreover, in this case~\ref{Ass:Gam-tildeGam}~\&~\ref{Ass:G} are satisfied at once.
However, in general the matrix~$\Gamma$ is not (or only partially) known and hence verification of conditions~\ref{Ass:Gam-tildeGam}~\&~\ref{Ass:G} causes problems.
In such general cases methods for parameter identification can be useful. 
For linear systems of type~\eqref{eq:linear-system} there is plenty of literature on system identification, see for instance~\cite{RaoSiva82,BingSinh90}, and the recent work~\cite{WaarPers20} where under the assumptions of controllability and persistently exciting inputs system identification is performed; in~\cite{ZhenLi21} the estimated parameters result from a least square problem.
Note that although the system identification in~\cite{WaarPers20,ZhenLi21} is developed for time-discrete linear systems the results can be applied to time-continuous linear systems to some extend, see e.g.~\cite[Sec.~1]{BingSinh90}.
In~\cite{FlorPont02} parameter identification for nonlinear systems is studied, where under a identifiability condition and with the aid of a high-gain observer system parameters are identified.
In~\cite{FarzMena15} an extended high-gain observer is introduced to identify the state and the unknown parameters dynamically; and
in the recent (rather technical) work~\cite{KaltNguy21} parameter identification via an adaption scheme for nonlinear systems is proposed and an error bound between the nominal and the estimated parameter is given.
However, the approaches~\cite{FlorPont02,FarzMena15,KaltNguy21} involve the system equations and hence the parameter identification is not model free.
Nevertheless, if a model is available an extension of~\cite[Prop. 2.1]{KaltNguy21} to matrix valued parameters may yield error bounds such that the stronger version of~\ref{Ass:G} 
\begin{equation*}
\|\tilde \Gamma^{-1}\| \|\tilde \Gamma - \Gamma\| < \min \left\{\frac{\rho -1}{r-2}, \frac{ \rho }{4\rho^2(\rho+1)^{r-2} - 1} \right\}
\end{equation*}
can ensured to be satisfied.
\end{Remark}

\begin{Remark} \label{Rem:k-derivatives-known}
If the first~$k \le r-1$ derivatives of the output signal~$y(t)$ are known, the funnel pre-compensator can be applied to~$y^{(k)}(t)$. 
Then, the condition~$T \in \cT^{{rm},q}_{\sigma,1}$ in Theorem~\ref{Thm:FPC-min-phase} becomes the relaxed condition~$T \in \cT^{{rm},q}_{\sigma,k}$.
Moreover, the error bound tightens and we have
\begin{equation*}
\forall\, t\ge0\, : \ \| y(t) - z(t) \| < (\rho + r-2-k)\vp(t)^{-1} - \ve.
 \end{equation*} 
\end{Remark}

\begin{Remark}
Although the funnel pre-compensator introduced in~\cite{BergReis18b} may take signals~$y$ and~$u$ with different dimensions, the system class~$\cN^{m,r}$ under consideration is restricted to systems where the input and output have the same dimension;
this comes into play when applying control schemes to the conjuntion of a minimum phase system with a cascade of funnel pre-compensators, see Section~\ref{Sec:Tracking}. 
However, a careful inspection of the proof of Theorem~\ref{Thm:FPC-min-phase} yields that an extension of Theorem~\ref{Thm:FPC-min-phase} to systems with different input ($u :\rp \to \R^m$) and output ($y: \rp \to \R^p$) dimensions is possible, if~$m > p$. 
In the case~$m < p$ (less inputs than outputs) with~$\rk \Gamma = \rk \tilde \Gamma = m$ one would require the matrix product~$\Gamma \tilde \Gamma^\dagger \in \R^{p \times p}$ ($\tilde \Gamma^\dagger$ denotes a pseudoinverse of~$\tilde \Gamma$) to be strictly positive definite; however, since $\rk(\Gamma \tilde \Gamma^\dagger) \le \min\{ m,p \} < p$ only positive semi-definiteness can be demanded which is not sufficient.
If~$ m > p$ (more inputs than outputs) the proof of Theorem~\ref{Thm:FPC-min-phase} can be adapted such that the statement is still true in the following two cases:
\begin{enumerate}[label = (\roman{enumi}), leftmargin=*]
\item Known $m-p$ entries of the input are set to zero, w.l.o.g. the last~$m-p$. 
Therefore, $\Gamma u(t) = \Gamma_p (u_1(t), \ldots, u_p(t))^\top$, where $\Gamma_p = \Gamma_p^\top \in \R^{p \times p}$ and sign definiteness is required.
Then, conditions~\ref{Ass:Gam-tildeGam}~\&~\ref{Ass:G} have to be satisfied for $\Gamma_p$ and $\tilde \Gamma_p$.
\item The system itself ignores known $m-p$ entries of the input (w.l.o.g. the last~$m-p$), i.e., $\Gamma = [\Gamma_p, 0]$, where $\Gamma_p = \Gamma_p^\top \in \R^{p \times p} $ and sign definiteness is required. 
Then, $\Gamma u(t) = \Gamma_p (u_1(t),\ldots,u_p(t))^\top$.
Choosing $\tilde \Gamma = [\tilde \Gamma_p, 0]$ with $\tilde \Gamma_p = \tilde \Gamma_p^\top > 0$, conditions~\ref{Ass:Gam-tildeGam}~\&~\ref{Ass:G} have to be satisfied for $\Gamma_p$ and $\tilde \Gamma_p$. 
\end{enumerate}
In both cases the respective transformations in \emph{Step 1} of the proof are feasible and the proof of Theorem~\ref{Thm:FPC-min-phase} can be done with corresponding matrices $\Gamma_p$ and $ \tilde \Gamma_p$.
\end{Remark}

\section{Output feedback control} \label{Sec:Tracking}
In this section we discuss the combination of the funnel pre-compensator with feedback control schemes.
Theorem~\ref{Thm:FPC-min-phase} yields that the conjunction of a minimum phase system~\eqref{eq:System} 
with a cascade of funnel pre-compensators~\eqref{eq:FPC-cascade} is again a minimum phase system and hence amenable to funnel control; for funnel control schemes for systems with higher relative degree see e.g.~\cite{IlchRyan07,BergLe18a} and the recent work~\cite{BergIlch21}.
We show that the combination of the funnel pre-compensator with a funnel control scheme achieves output tracking with prescribed transient behaviour of the tracking error via output feedback only.
This resolves the long-standing open problem that for the application of funnel controller to systems with higher relative degree the derivatives of the output are required to be known.

\ \\
Compared to the control scheme in~\cite{BergLe18a} - involving high derivatives of virtual error variables which need to be calculated before implementation - the control scheme from~\cite{BergIlch21} is much easier to implement.  
Therefore, along with other benefits (see Remark~\ref{Rem:aspects-FC} below) we recall the funnel control scheme from~\cite[Thm. 1.9]{BergIlch21}.
For $e(t) := z(t) - y_{\rm ref}(t) \in \R^m$ define $\textbf{e}(t) := (e^{(0)}(t)^\top,\ldots,e^{(r-1)}(t)^\top )^\top \in \R^{rm} $, this is, the instant error vector between the signal~$z$ and the reference signal~$y_{\rm ref}$ and their derivatives, respectively.
Next, the control parameters are chosen. The funnel function~$\phi$ belongs to the set
\begin{equation*}
\Phi_{\rm FC} := \setdef{ \phi \in \text{AC}_{\rm loc}(\rp \to \R)}{
\begin{array}{l}
\forall\,s >0 \,:\ \phi(s)>0, \ \liminf_{s \to \infty} \phi(s) >0, \\
\exists\, c>0\,:\ |\dot \phi(s)| \le c \, (1+\phi(s)) \text{ for a. a. } s\ge 0 
\end{array} }, 
\end{equation*}
where~$\text{AC}_{\rm loc}(\rp \to \R)$ denotes the set of locally absolutely continuous functions~$f:\rp \to \R$; 
$N \in \cC(\rp \to \R)$ is a surjection, $\alpha \in \cC^1([0,1), [1,\infty))$ is a bijection, and for $\cB := \setdef{w \in \R^m}{\|w\| < 1}$ the function ${\gamma : \cB \to \R^m, \ w \mapsto \alpha(\|w\|^2)w}$ is defined.
Further, recursively the maps $\rho_k : \cD_k \to \cB$, $k=1,\ldots,r$ are defined as follows
\begin{equation*}
\begin{aligned}
\cD_1 &:= \cB, \ \rho_1: \cD_1 \to \cB, \ \eta_1 \mapsto \eta_1, \\
\cD_k &:= \setdef{ (\eta_1^\top,\ldots,\eta_k^\top)^\top \in \R^{km}}{ 
\begin{array}{l}
(\eta_1^\top,\ldots,\eta_{k-1}^\top)^\top \in \cD_{k-1}, \\
\eta_k + \gamma \left(\rho_{k-1}(\eta_1^\top,\ldots,\eta_{k-1}^\top)^\top \right) \in \cB
\end{array}}, \\
\rho_k &:  \cD_k \to \cB, \ (\eta_1^\top,\ldots,\eta_k^\top)^\top \mapsto \eta_k + \gamma \left(\rho_{k-1}(\eta_1^\top,\ldots,\eta_{k-1}^\top)^\top\right).
\end{aligned}
\end{equation*}
Then, the funnel control scheme from~\cite[Sec. 1.4, Eqt. (9)]{BergIlch21} is given by
\begin{equation} \label{eq:FC}
u(t) = (N \circ \alpha)(\|w(t)\|^2) w(t), \ w(t) := \rho_r(\vp(t) \textbf{e}(t)).
\end{equation}
\begin{Remark}
Since by~\ref{Ass:Gam-tildeGam} we have~$\tilde \Gamma > 0$ according to~\cite[Rem.~1.8.(b)]{BergIlch21} the simple choice~$N(s) = -s$ is feasible.
Moreover, similar to the gain functions~$h_i$ in~\eqref{eq:FPC-cascade} we may choose~$\alpha(s) = 1/(1-s)$ by which the control scheme is given by
\begin{equation} \label{eq:FC-simple}
u(t) = - \frac{w(t)}{1 - \|w(t)\|^2}, \quad w(t) := \rho_r(\vp(t) \textbf{e}(t)).
\end{equation}
\end{Remark}

\ \\
Now, if the reference trajectory satisfies
\begin{equation} \label{eq:regularity_ref}
y_{\rm ref} \in \cW^{r,\infty}(\rp \to \R^m), 
\end{equation}
we have the following result.
\begin{Corollary} \label{Cor:Output-tracking-via-FPC}
Consider a system~\eqref{eq:System} with~$(d,f,T,\Gamma) \in \cN^{m,r}_1$ and $y^0 \in \cW^{r-1,\infty}([-\tau,0] \to \R^m)$ in conjunction with a cascade of funnel pre-compensators~\eqref{eq:FPC-cascade} with~$(a,p,\vp,\vp_1,\rho,\tilde \Gamma) \in \Sigma$. 
Furthermore, assume the initial conditions~\eqref{eq:inital} in Theorem~\ref{Thm:FPC-min-phase} are satisfied.
Moreover, let $\phi \in \Phi_{\rm FC}$ and assume that for the pre-compensator's output~$z := z_{r-1,1}$ the funnel control initial value constraint 
\begin{equation*}
\phi(0) \textbf{e}(0) \in \cD_r, \ e(t) := z(t) - y_{\rm ref}(t), \ \textbf{e}(t) := (e(t)^\top,\ldots,e^{(r-1)}(t)^\top)^\top
\end{equation*}
is satisfied.
Then, if~\eqref{eq:regularity_ref} is satisfied, the funnel controller~\eqref{eq:FC}, with input $z = z_{r-1,1}$ from~\eqref{eq:FPC-cascade}, applied to system~\eqref{eq:System} yields an initial-value problem, which has a solution, and every solution can be extended to a maximal solution
$(y,\zeta): [-\tau,\omega) \to \R^{m} \times \R^{rm(r-1)}$, where $\omega \in (0,\infty]$ and $\zeta=(z_{1,1}^\top,\ldots,z_{r-1,r}^\top)^\top \in \R^{m(r-1)r}$, and the maximal solution has the properties
\begin{itemize}
\item[i)] the solution is global, this is $\omega = \infty$,
\item[ii)] 
the input~$u$, 
the compensator states~$\zeta$, the pre-compensator gain functions $h_1,\ldots,h_{r-1}$, and the original system's output and its derivatives $y,\dot y,\ldots,y^{(r-1)}$ are bounded, this is, for all $i=1,\ldots,r-1$ we have
$u \in \cL^\infty(\rp \to \R^m)$, 
$\zeta \in \cL^\infty(\rp \to \R^{m(r-1)r})$,
$h_i \in \cL^\infty(\rp \to \R)$,
$y \in \cW^{r, \infty}(\rp \to \R^m)$,
\item[iii)] the errors evolve in their respective performance funnels, this is
\begin{equation*}
\begin{aligned}
\exists\, \ve_1 > 0 \ \forall\, t > 0: \ & \|y(t) - z_{1,1}(t) \| < \vp_1(t)^{-1} - \ve_1 , \\
\forall\, i=2,\ldots,r-1 \ \exists\, \ve_i > 0 \ \forall\, t > 0: \ & \|z_{i-1,1}(t) - z_{i,1}(t) \| < \vp(t)^{-1} - \ve_i , \\
\exists\, \beta > 0 \ \forall\, t > 0: \ & \|z(t) - y_{\rm ref}(t) \| < \phi(t)^{-1} - \beta.
\end{aligned}
\end{equation*}
\end{itemize}
In particular, with $\ve := \sum_{i=1}^{r-1} \ve_i$, the tracking error~$y - y_{\rm ref}$ evolves within a prescribed funnel, this is
\begin{equation*}
\forall \, t > 0: \ \|y(t) - y_{\rm ref}(t) \| < (\rho+r-2)\vp(t)^{-1} + \phi(t)^{-1} - (\ve + \beta) . 
\end{equation*}
\end{Corollary}
\begin{proof}
Theorem~\ref{Thm:FPC-min-phase} yields that a system~\eqref{eq:System} with~$(d,f,T,\Gamma) \in \cN^{m,r}_1$ in conjunction with a cascade of funnel pre-compensators~\eqref{eq:FPC-cascade} with $(a,p,\vp,\vp_1,\rho,\tilde \Gamma) \in \Sigma$
results in a minimum phase system with~$(\tilde d, \tilde F, \tilde T, \tilde \Gamma) \in \cN^{m,r}$, and the respective aspects of assertions~$ii)~\&~iii)$, namely concerning $\zeta, h_i$, are true; hence the conjunction belongs to the system class under consideration in~\cite{BergIlch21}.
Then~\cite[Thm. 1.9]{BergIlch21} is applicable and yields the remaining aspects of assertions~$i)-iii)$.
Furthermore, the transient behaviour of the tracking error~$y-y_{\rm ref}$ is a direct consequence of~$iii)$.
\qed
\end{proof}

\begin{Remark} 
Recalling Remark~\ref{Rem:linear-min-phase-satisfy-a-prime} we highlight that Corollary~\ref{Cor:Output-tracking-via-FPC} guarantees for a large class of systems output tracking with prescribed transient behaviour of the tracking error 
via output feedback only; note that this result applies to single-input, single-output systems as well as to multi-input, multi-output systems.
The controller proposed in~\cite{DimaBech20} as well achieves output tracking with prescribed performance of the error via output feedback for minimum phase multi-input, multi-output systems of arbitrary relative degree; in particular, this controller as well as the control scheme~\eqref{eq:FC} is applicable to linear minimum phase systems. 
However, as it involves a high-gain observer structure, the controller from~\cite{DimaBech20} suffers from the problem of proper initializing, i.e., some parameters have to be chosen large enough in advance, however, it is not clear how large.
Contrary, conditions~\ref{Ass:A-Hurwitz}~--~\ref{Ass:G} explicitly determine the set~$\Sigma$ of feasible design parameters of the funnel pre-compensator.
This resolves a long-standing problem in the field of high-gain based output feedback control with prescribed transient behaviour.
\end{Remark}

\begin{Remark}
An application of the funnel pre-compensator to linear non-minimum phase system under consideration in~\cite{Berg20} allows output feedback tracking for a certain class of linear non-minimum phase systems with the controller scheme proposed in~\cite[Sec.~3]{Berg20}.
However, in combination with the funnel pre-compensator the bound of the tracking error discussed in~\cite[Sec.~4]{Berg20} is not valid any more.
For deeper insights regarding output tracking of linear non-minimum phase systems see~\cite{Berg20}.
\end{Remark}

\begin{Remark} \label{Rem:aspects-FC}
We highlight two important aspects.
First, according to~\cite[Sec. 1.4]{BergIlch21}, in particular~\cite[Rem. 1.7 (c)]{BergIlch21} tracking of a given reference is also possible, if the number of the available derivatives of the reference signal is smaller than the relative degree. This means the following. Let~$\hat r $ be the number of derivatives of the reference signal~$y_{\rm ref}$ available for the feedback controller. Then with the control scheme~\eqref{eq:FC} output tracking with prescribed transient behaviour of the tracking error is possible in the case~$\hat r < r$.
Second, in the case~$\hat r = r$ exact asymptotic tracking can be achieved, see~\cite[Rem.~1.7 (f)]{BergIlch21}.
In the present context this means $\lim_{t \to \infty} (z(t)-y_{\rm ref}(t)) = 0$, however, $\lim_{t \to \infty} (y(t) - z(t)) = 0$ cannot be guaranteed since~$\vp \in \Phi_r$. 
Note that, while the second aspect is of limited practical interest (see also~\cite[Sec.~1.1]{BergIlch21}), the first aspect allows, for instance, target tracking of a given ``smooth" trajectory, the derivatives of which are as unknown as the derivatives of the system.
\end{Remark}

\begin{Remark}
Remark~\ref{Rem:k-derivatives-known} applies also for output tracking, i.e., if the first~${k \le r-1}$ derivatives of the output signal~$y(t)$ are known, the pre-compensator takes the signal~$y^{(k)}(t)$ which results in a tighter funnel boundary of the tracking error~${y - y_{\rm ref}}$. 
Moreover, in this case we only require~$T \in \cT^{{rm},q}_{\sigma,k}$, this is, the funnel pre-compensator is applied to systems with~$(d,f,T,\Gamma) \in \cN^{m,r}_k$.
\end{Remark}

\section{Simulations} \label{Sec:Simulations}
In this section we provide simulations of the funnel pre-compensator and its applications. 
First, we consider the pure functionality of the funnel pre-compensator applied to signals~$(u,y) \in \cP_r$.
Second, we simulate output tracking via funnel control. 
To this end, we apply a funnel control scheme to the conjunction of a cascade of funnel pre-compensators with a minimum phase system.

\begin{Example} \label{Ex:functionality}
We give an illustrative example how the funnel pre-compensator works. 
Moreover, we qualitatively compare the influence of the pre-compensator's parameters to its performance.
We choose the signals
\begin{equation*}
y(t) = e^{-(t-5)^2},
\quad u(t) = \sin(t),
\end{equation*}
and, since $(u,y) \in \cC^\infty(\rp \to \R) \times \cC^\infty(\rp \to \R)$, we choose to simulate the application of the cascade of funnel pre-compensators~\eqref{eq:FPC-cascade} for the case~$r=3$.
For the funnel pre-compensator we initially choose a Hurwitz polynomial, i.e., a polynomial whose roots have strict negative real part,
by which the matrix~$A \in \R^{3 \times 3}$ from~\ref{Ass:A-Hurwitz} is determined via the polynomial's coefficients, cf. Remark~\ref{Rem:Constants}.
Since it turned out that the parameters~$a_i$ influence the pre-compensator's performance the most we compare three sets of these parameters. 
To this end, we choose the Hurwitz polynomials
\begin{equation*}
\begin{aligned}
(\alpha + 1)^3 &= \alpha^3 + 3\alpha^2 + 3\alpha + 1, \\
(\beta + 3)^3 &= \beta^3 + 9 \beta^2 + 27\beta + 27,\\
(\gamma + 5)^3 & = \gamma^3 + 15\gamma^2 + 75\gamma + 125
\end{aligned}
\end{equation*}
by which we obtain corresponding matrices
\begin{equation*}
A = \begin{bmatrix}
-3 & 1 & 0 \\
- 3 & 0 & 1 \\
-1 & 0 & 0
\end{bmatrix}, \qquad
B = \begin{bmatrix}
-9 & 1 & 0 \\
- 27 & 0 & 1 \\
-27 & 0 & 0
\end{bmatrix}, \qquad
C= \begin{bmatrix}
-15 & 1 & 0 \\
- 75 & 0 & 1 \\
-125 & 0 & 0
\end{bmatrix}
\end{equation*}
which means~$a_1 = 3$, $a_2 = 3$, $a_3 = 1$, 
$b_1 = 9$, $b_2 = 27$, $b_3 = 27$,
and $c_1 = 15$, $c_2 = 75$, $c_3 = 125$.
Choosing~$Q := I_3$ the respective Lyapunov matrix~$P_i$, ${i\in \{A,B,C\}}$ is given as
\begin{equation*}
P_A = \begin{bmatrix}
1 & -\frac{1}{2} & -1 \\
-\frac{1}{2} & 1 & -\frac{1}{2} \\
-1 & -\frac{1}{2} & 4
\end{bmatrix}, \qquad
P_B = \begin{bmatrix}
       4     &        -\frac{1}{2} & -\frac{22}{27} \\    
       -\frac{1}{2}  & \frac{22}{27} &        -\frac{1}{2}    \\
     -\frac{22}{27}    &       -\frac{1}{2}    &       \frac{61}{81} 
\end{bmatrix} \quad
P_C = \begin{bmatrix}
\frac{58}{5} & -\frac{1}{2} & -\frac{136}{125} \\
- \frac{1}{2} & \frac{136}{125} & -\frac{1}{2} \\
-\frac{136}{125} & -\frac{1}{2} & \frac{1333}{3125}
\end{bmatrix},
\end{equation*}
from which~$p_1^a = 1$, $p_2^a = 2/3$, $p_3^a = 1/3$ ,
$p_1^b = 1$, $p_2^b = 1037/481$, $p_3^b = 1787/711$,
and~$p_1^c = 1$, $p_2^c = 1383/391$, $p_3^c = 2230/333$.
Further, we choose the funnel function~$\vp(t) = (e^{-2t} + 0.05)^{-1}$, and~$\vp_1(t) = \vp(t)/\rho$ for~$\rho = 1.5$.
Finally, we choose $z_{i,j}(0) = 0$ for all~$i=1,2$, $j=1,2,3$ by which the conditions on the initial values in~\cite[Thm.~1]{BergReis18b} are satisfied.
We run the simulation over the time interval~$0 - 10$ seconds.
The outcome is depicted in Figure~\ref{Fig:FPC-to-Sig}.
\begin{figure}[h!]
\begin{subfigure}[t]{0.49\textwidth}
\includegraphics[width=\textwidth]{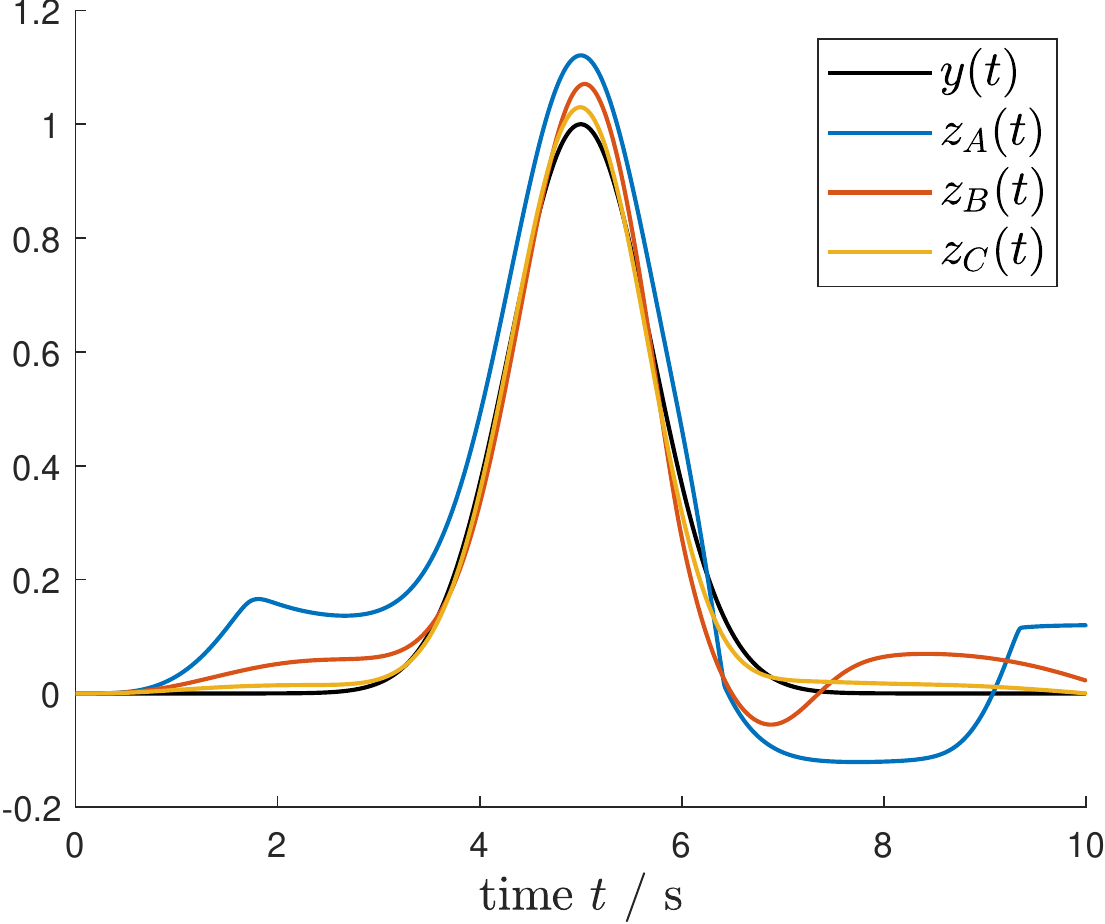}
\caption{Signal~$y(t)$ and the pre-compensator's output for different choices of parameters, respectively.}
\label{Fig:FPC-to-Sig-States}
\end{subfigure}
\hfill
\begin{subfigure}[t]{0.49\textwidth}
\includegraphics[width=\textwidth]{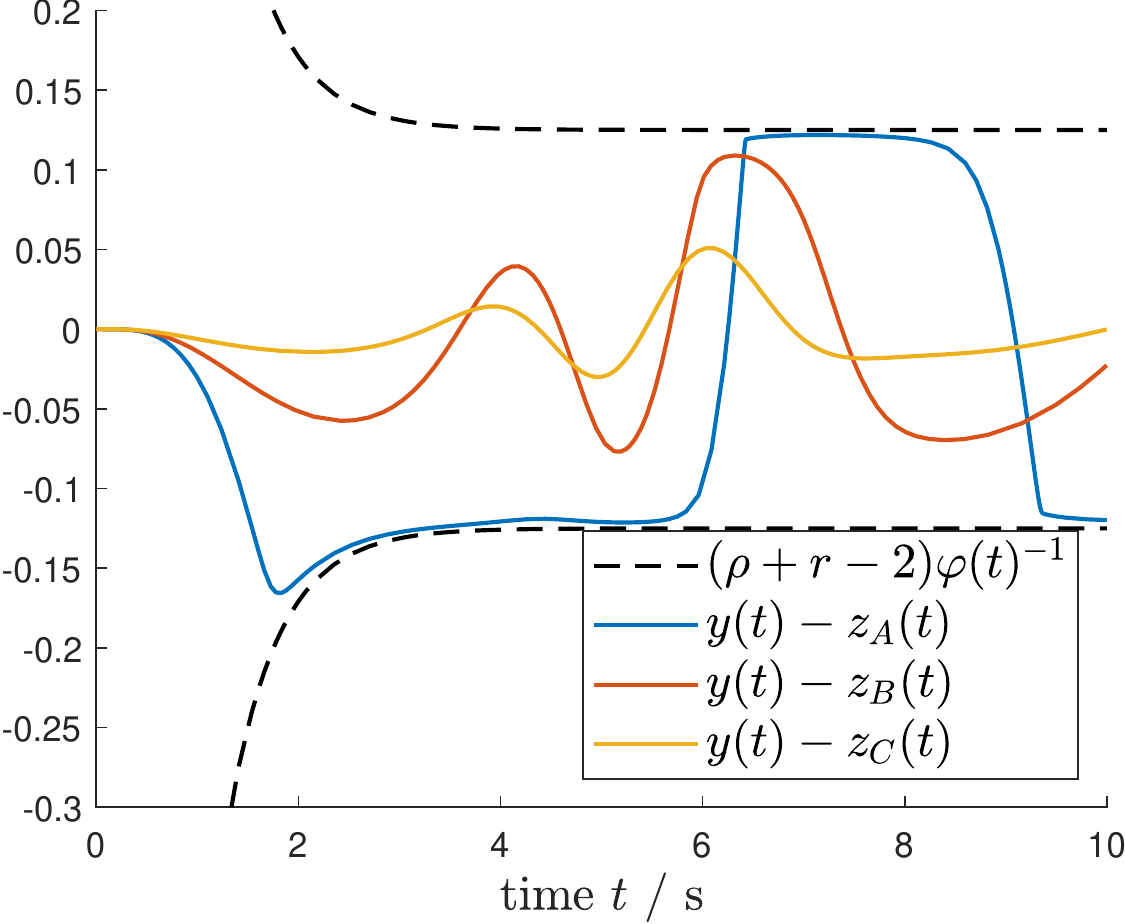}
\caption{Errors between the signal~$y(t)$ and the pre-compensator's output for different choices of parameters, respectively.}
\label{Fig:FPC-to-Sig-Errors}
\end{subfigure}
\caption{Simulation of functionality of the funnel pre-compensator applied to given signals~$(u,y) \in \cP_3$.}
\label{Fig:FPC-to-Sig}
\end{figure}
Figure~\ref{Fig:FPC-to-Sig-States} shows the signal~$y$ and the pre-compensator's output~$z_A, z_B, z_C$, respectively.
The ``quality" of the approximation depends strongly on the choice of the parameters~$a_i$ satisfying~\ref{Ass:A-Hurwitz}.
While in the first case the signal~$y$ and~$z_A$ differ quite much, in the second case the approximation is much better, and in the third case the signal~$y$ and the output~$z_C$ are almost identical which means that the approximation of the given signal by the funnel pre-compensator is pretty good.
Figure~\ref{Fig:FPC-to-Sig-Errors} shows the error between the signal~$y$ and the pre-compensator's output, respectively.
Here the aforesaid crystallises from the viewpoint of errors, which are quite different. 
However, we highlight that in all three cases the error evolves within the prescribed funnel boundaries and hence all approximations~$z_A, z_B, z_C$ of the signal~$y$ are at least as good as a ``predetermined quality".
The simulation has been performed in \textsc{Matlab} (solver: \textsf{ode23tb}).
\end{Example}

\ \\
Next, we apply the funnel control scheme~\eqref{eq:FC-simple} to the conjunction of a cascade of funnel pre-compensators~\eqref{eq:FPC-cascade} with a minimum phase system~\eqref{eq:System} 
to achieve output tracking with prescribed transient behaviour of the tracking error via output feedback only, i.e., we illustrate an application of Corollary~\ref{Cor:Output-tracking-via-FPC}.
We emphasize that the funnel pre-compensator receives only the measurement of the output signal~$y$ of the system under consideration.
Then, the applied controller~\eqref{eq:FC-simple} takes the pre-compensator's output~$z$ and its derivatives which are known explicitly.

\begin{Example}  \label{Ex:tracking}
Since the application of the funnel pre-compensator to the standard illustrative example \textit{mass on a car system} from~\cite{SeifBlaj13} was already discussed in detail in~\cite[Sec. 5.1]{BergReis18b}, and the application of the controller~\eqref{eq:FC} to this particular example was elaborated in~\cite[Sec. 3.1]{BergIlch21}, we consider the following artificial, in particular, nonlinear multi-input multi-output ODE of relative degree~$r=3$ with~$m=2$ and initial conditions $y|_{[-\tau,0]} \equiv 0 \in \R^2$ for some~$\tau >0$, 
\begin{equation} \label{eq:Sim:System}
\begin{aligned}
y^{(3)}(t) &= R_1 y(t) + R_2 \dot y(t) + R_3 \ddot y(t) +  f\left( d(t), T(y, \dot y, \ddot y)(t) \right) + \Gamma u(t), 
\end{aligned}
\end{equation}
where
\begin{equation*}
 \begin{array}{llll}
R_1 = \begin{bmatrix}
-1 & 0 \\ 0 & 0
\end{bmatrix}, &
R_2 = \begin{bmatrix}
1 & -1 \\ 0 & 0
\end{bmatrix}, &
R_3 = \begin{bmatrix}
1 & 1 \\ 0 & -1
\end{bmatrix}, &
\Gamma = \begin{bmatrix}
2 & 0.2\\ 0.2 & 2
\end{bmatrix} = \Gamma^\top > 0  ,
 \end{array}
\end{equation*} 
and for~$d = (d_1,d_2)^\top$,~$\xi_{i} = (\xi_{i,1},\xi_{i,2})^\top$, $i=1,2,3$
\begin{equation*}
\begin{aligned}
T : \cC([-\tau,\infty) \to \R^6)   &\to \cL_{\rm loc}^{\infty}(\rp \to \R^3), \\
 (\xi_1(\cdot), \xi_2(\cdot), \xi_3(\cdot))  &\mapsto 
\left(t \mapsto
 \begin{pmatrix}
\xi_{1,1}(t)^2 + e^{\xi_{1,1}(t) - |\xi_{2,1}(t)|}  \\
\xi_{1,2}(t)^3  - \sin(\xi_{2,2}(t))  \\
\int_{0}^{t} e^{-(t-s)}\|\xi_1(t)\|^2\tanh(\|\xi_3(t)\|^2) \, \ds 
\end{pmatrix} \right), \\
f: \R^2 \times \R^3  & \to \R^2, \\ 
(d_1,d_2, \zeta_1, \zeta_2 , \eta) & \mapsto \begin{pmatrix}
d_1 + \zeta_1 + \eta^3 \\ d_2 + \zeta_2 - \eta
\end{pmatrix} \\
\end{aligned}
\end{equation*}
whereby the internal dynamics are bounded-input bounded-state stable and the associated operator~$T$ belongs to the class~$\cT^{6,3}_{\tau,1}$.
The disturbance is chosen 
as~$d : \rp \to \R^2$, $ t \mapsto (0.2 \sin(5t) + 0.2 \cos(7t), \, 0.25 \sin(9t) + 0.2 \cos(3t))^\top$
by which $d \in \cL^\infty(\rp \to \R)$ and hence~$(d,f,T,\Gamma) \in \cN^{2,3}_1$.
For the funnel pre-compensator we choose the Hurwitz polynomial~$(s+s_0)^3$ with~$s_0 = 7$  
and~$Q = I_3$ to determine matrices~$A$ and~$P$ satisfying~\ref{Ass:A-Hurwitz} and obtain the respective 
parameters~$a_1 = 21$, $a_2 = 147$, $a_3 = 343$
and~$p_1 = 1$, $p_2 = 1180/241$, $p_3 = 1742/135$. 
We choose the pre-compensator's funnel function
$\vp(t) = (e^{-3t} + 0.05 )^{-1}$, and for~$\tilde \Gamma = 2 \cdot I_2$ with~$\rho = 1.1$ conditions~\ref{Ass:Gam-tildeGam}~\&~\ref{Ass:G} are satisfied. 
We stress that~$\Gamma$ and~$\tilde \Gamma$ are quite different; while $\Gamma$ distributes both input signals~$(u_1,u_2)$ to both output directions~$(y_1,y_2)$, $\tilde \Gamma$ only allocates the input signal~$u_i$ to output direction~$y_i$, $i=1,2$.
Next, we choose the controller's funnel function~$\phi_{\rm FC}(t) = (2e^{-t} + 0.05 )^{-1}$.
Then, with $z_{i,j}(0) = 0$, $i=1,2$, $j=1,2,3$ the assumptions on the initial values from Corollary~\ref{Cor:Output-tracking-via-FPC} are satisfied.
Further, we choose the reference trajectory as $y_{\rm ref}: \rp \to \R^2$, $t \mapsto (e^{-(t-5)^2}, \, \sin(t) )^\top$.
We simulate the output tracking over a time interval~$0-10$ seconds. 
In order to illustrate the funnel pre-compensator's contribution we compare the two cases, first, if the derivatives of the output of system~\eqref{eq:Sim:System} are available to the controller and second, if not.
The outcomes of the simulations are depicted in Figure~\ref{Fig:FPC-Track}, where the output's subscript~$\rm FC$~($y_{\rm FC}$) denotes the case when the derivatives of the system are available, i.e., the funnel pre-compensator is not necessary and hence not present; and the subscript~$\rm FPC$~($y_{\rm FPC}$) indicates the situation when the system's output is approximated by the pre-compensator and the derivatives of the latter are handed over to the controller.
\begin{figure}[h!]
\begin{subfigure}[t]{0.49\textwidth}
\includegraphics[width=\textwidth]{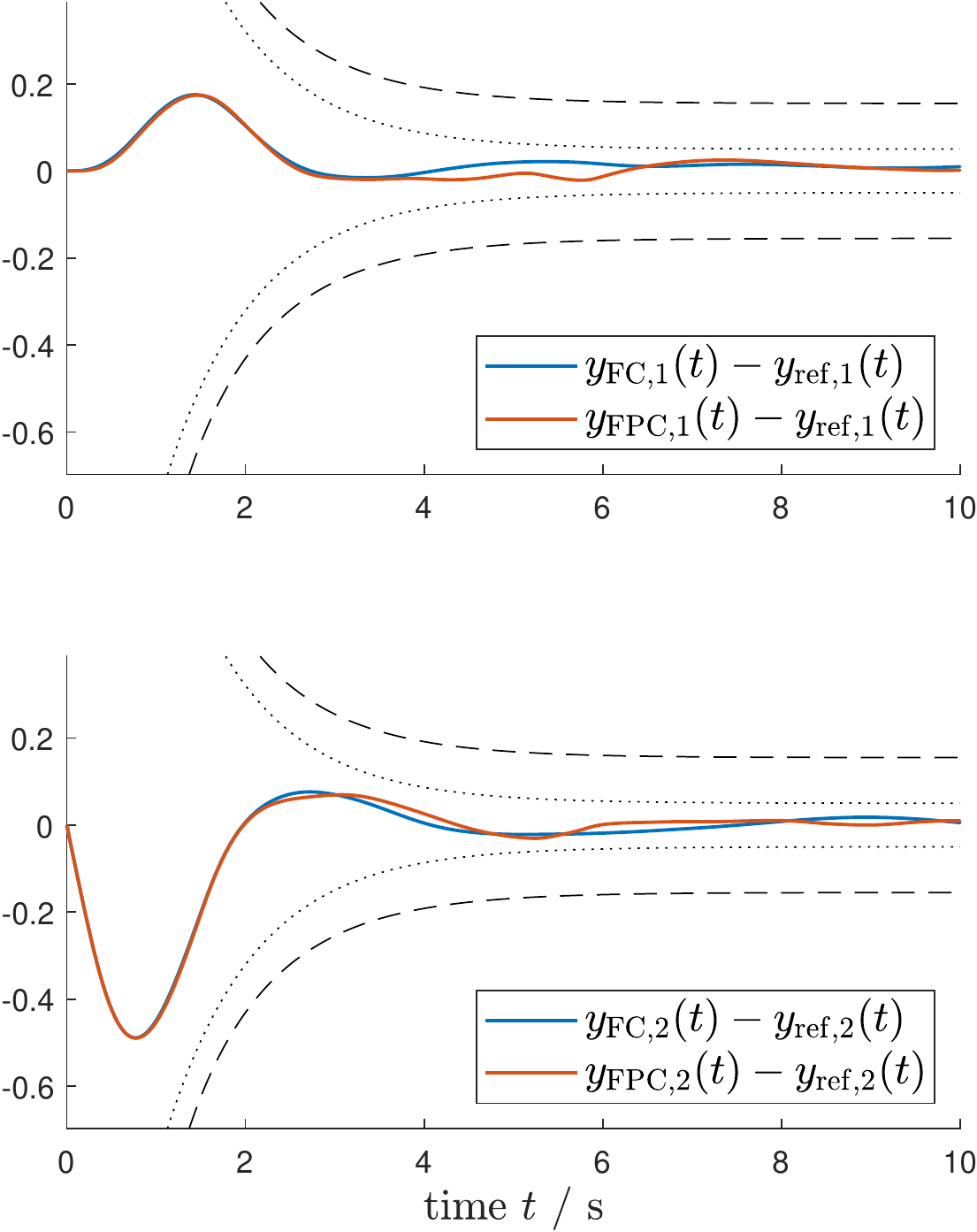}
\caption{Tracking error~$y(t) - y_{\rm ref}(t)$. The dashed line~$- - -$ represents the funnel boundary given by $(\rho + r - 2)/\vp(t) + 1/\phi_{\rm FC}(t)$, the dotted line~$\cdots$ represents the funnel boundary given by~$1/\phi_{\rm FC}(t)$.}
\label{Fig:FPC-Track-Error}
\end{subfigure}
\hfill 
\begin{subfigure}[t]{0.49\textwidth}
\includegraphics[width=\textwidth]{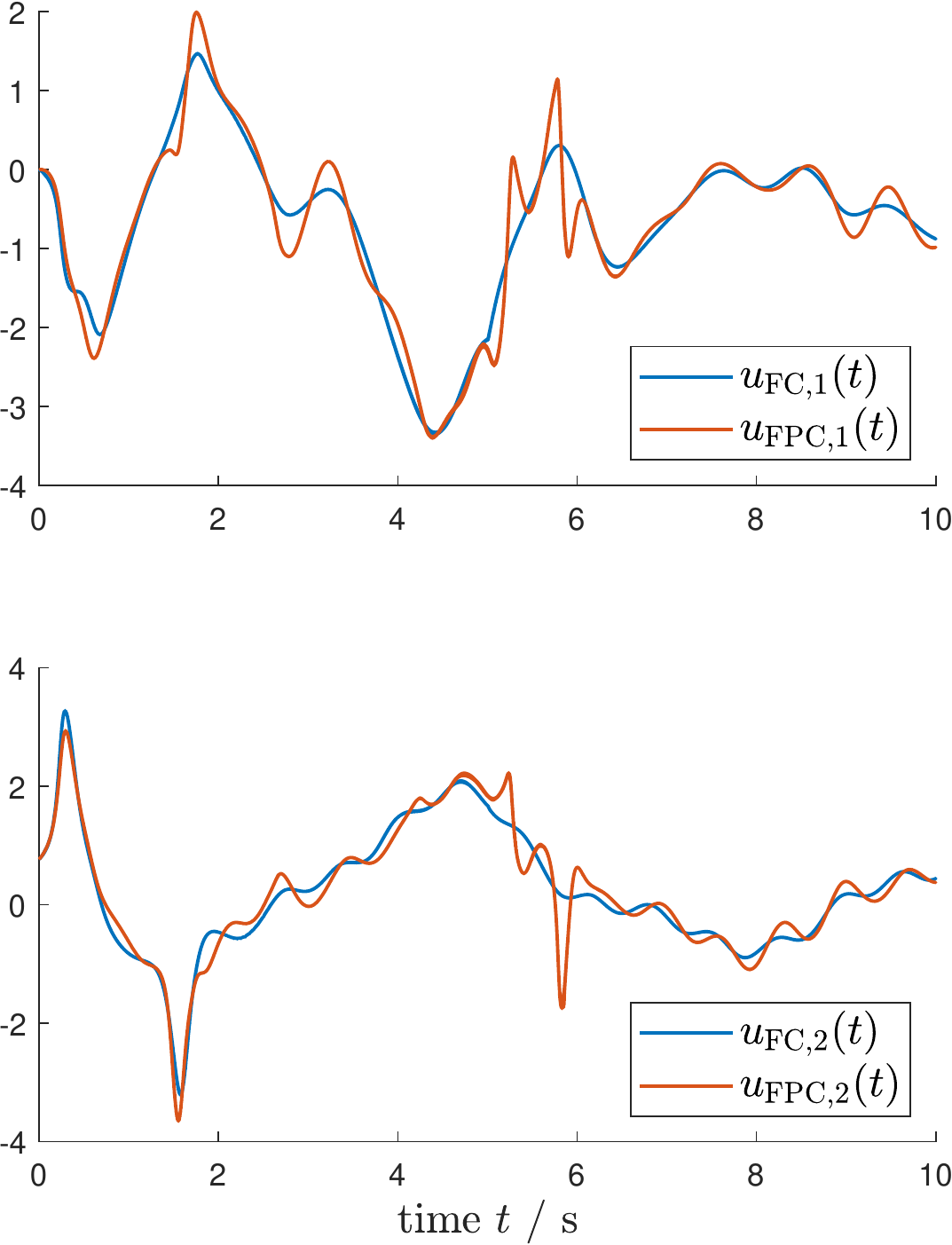}
\caption{Input signal~$u$ generated by the funnel control scheme~\eqref{eq:FC-simple}.}
\label{Fig:FPC-Track-Control}
\end{subfigure}
\caption{Output reference tracking of the nonlinear multi-input multi-output minimum phase system~\eqref{eq:Sim:System} via output feedback.}
\label{Fig:FPC-Track}
\end{figure}
Figure~\ref{Fig:FPC-Track-Error} shows the tracking error between the system's output~$y_i$ and the reference trajectory~$y_{\rm ref,i}$, $i=1,2$, in both cases, respectively.
Note that in the case when the derivatives of the system's output are available the error evolves within the funnel boundaries defined by~$1/\phi_{FC}$ as expectable from the results in~\cite[Thm. 1.9]{BergIlch21}.
We emphasize that in the second case (the derivatives of the system's output are not available) the transient behaviour of the tracking error can be guaranteed to evolve within the boundaries $ (\rho + r-2)/\vp + 1/\phi_{\rm FC} = {(\rho+1)/\vp + 1/\phi_{\rm FC}}$, where~$r=3$ and~$\rho=1.1$ in this particular example.
Moreover, we can observe that even in the latter case the error evolves within the boundaries of~$1/\phi_{\rm PC}$.
Figure~\ref{Fig:FPC-Track-Control} shows the control input~$u_i$, $i=1,2$, generated by~\eqref{eq:FC-simple}, respectively.
Both signals show oscillations which arise from the influence of the disturbance~$d$, i.e., the controller compensates the disturbance's affect to the system.
Analogously to the approximation performance discussed in the previous Example~\ref{Ex:functionality} the tracking performance (and so the input signal~$u$) strongly depends on the choice of the pre-compensator's parameters, this is, on the choice of~$A$ (for fixed~$Q = I_m$) and can be improved with larger values of~$s_0$. 
The simulations have been performed in \textsc{Matlab} (solver: \textsf{ode45}, rel. tol: $10^{-6}$, abs. tol: $10^{-6}$).
\end{Example}

\section{Conclusion}
In the present article we showed that the conjunction of the funnel pre-compensator with a minimum phase system of arbitrary relative degree results in a minimum phase system of the same relative degree. This resolves the open question raised in~\cite{BergReis18b} where the funnel pre-compensator was introduced and the aforesaid was proven for the special case of relative degree two.
Using the fact that the derivatives of the pre-compensator's output are known explicitly we showed that output reference tracking with prescribed transient behaviour using funnel based feedback control schemes is possible with output feedback only.
In particular, output tracking with unknown output derivatives is possible for the class of linear minimum phase systems~\eqref{eq:linear-system}; and moreover, for a class of linear non-minimum phase systems (single-input, single output systems as well as multi-input, multi-output systems).
Since the investigations in the recent works~\cite{BergLanz20} and~\cite{BergDrue21} show applicability of existing control techniques to nonlinear non-minimum phase systems we are confident that an integration of the funnel pre-compensator into this particular context will also be fruitful.

\ \\
Now, in future research we will investigate the extension of the present results to a larger class of systems, in particular, systems with~$(d,f,T,\Gamma) \in \cN^{m,r}$  will be focused; furthermore, systems which are not linear affine in the control term will be investigated.

\begin{acknowledgements}
I am deeply indebted to Thomas Berger (University of Paderborn) for excellent mentoring, many fruitful discussions, helpful advices and corrections.
\end{acknowledgements}

\section*{Appendix}
\paragraph{Proof of Theorem~\ref{Thm:FPC-min-phase}.}
Since the special case~$r=2$ was already proven in~\cite[Thm.~2]{BergReis18b}, we concentrate on the case~$r \ge 3$. 
Therefore, in the following let~$r \ge 3$.
%
The proof is subdivided in three steps.

\ \\
\underline{Step 1}: We briefly present the transformations performed in~\cite[p. 4758-4760]{BergReis18b}. 
We consider the error dynamics of two successive systems. Set $v_{i,j}(\cdot) := z_{i-1,j}(\cdot) - z_{i,j}(\cdot)$ for $i=2,\ldots,r-1$ and $j=1,\ldots,r$. Then,
\begin{equation*}
\begin{aligned}
\dot v_{i,1}(t) &= v_{i,2}(t) - (a_1 + p_1 h_i(t)) v_{i,1}(t) + (a_1 + p_1 h_{i-1}(t))v_{i-1,1}(t) , \\
& \ \, \vdots \\
\dot v_{i,r-1}(t) &= v_{i,r}(t) - (a_{r-1} + p_{r-1} h_{i}(t))v_{i,1}(t) + (a_{r-1} + p_{r-1}h_{i-1}(t))v_{i-1,1}(t) , \\
\dot v_{i,r}(t) &= -(a_r + p_r h_i(t)) v_{i,1}(t) + (a_r + p_r h_{i-1}(t)) v_{i-1,1}(t).
\end{aligned}
\end{equation*}
In order to investigate the dynamics of $v_{1,1}$, we define $e_{1,j}(\cdot) := y^{(j-1)}(\cdot) - z_{1,j}(\cdot)$ for $j=1,\ldots,r-1$, 
and $e_{1,r}(\cdot) := y^{(r-1)}(\cdot) - \Gamma \tilde \Gamma^{-1} z_{1,r}(\cdot)$. Then we obtain
\begin{equation*}
\begin{aligned}
\dot e_{1,1}(t) &= e_{1,2}(t) - (a_1 +  p_1 h_1(t))e_{1,1}(t) , \\
& \ \, \vdots \\
\dot e_{1,r-2}(t) &= e_{1,r-1}(t) - (a_{r-2} + p_{r-2} h_1(t))e_{1,1}(t) , \\
\dot e_{1,r-1}(t) &= e_{1,r}(t) - (a_{r-1} + p_{r-1} h_1(t))e_{1,1}(t) + (\Gamma \tilde \Gamma^{-1} - I_m) z_{1,r}(t) , \\
\dot e_{1,r}(t) &= -\Gamma \tilde \Gamma^{-1}(a_r + p_r h_1(t))e_{1,1}(t) \\
& \qquad + \sum_{i=1}^{r} R_i y^{(i-1)}(t) + f\big( d(t),T(y,\dot y,\ldots,y^{(r-1)})(t) \big).
\end{aligned}
\end{equation*}
Now, we set $v_{1,1}(\cdot) := e_{1,1}(\cdot)$, and define $\tilde v(\cdot) := \sum_{i=1}^{r-1} v_{i,1}(\cdot)$ 
and $v_{1,j}(\cdot) := e_{1,j}(\cdot) - \sum_{k=1}^{j-1} R_{r-j+k+1} \tilde v^{(k-1)}(\cdot)$ 
for $j=2,\ldots,r$. Then we obtain
\begin{equation*}
\begin{aligned}
\dot v_{1,1}(t) &= v_{1,2}(t) - (a_1 + p_1 h_1(t)) v_{1,1}(t) + R_r \tilde v(t) , \\
& \ \, \vdots \\
\dot v_{1,r-2}(t) &= v_{1,r-1}(t) - (a_{r-2} + p_{r-2} h_1(t)) v_{1,1}(t) + R_3 \tilde v(t) , \\
\dot v_{1,r-1}(t) &= v_{1,r}(t) - (a_{r-1} + p_{r-1} h_1(t)) v_{1,1}(t) + R_2 \tilde v(t) + (\Gamma \tilde \Gamma^{-1} - I_m) z_{1,r}(t)  , \\
\dot v_{1,r}(t)  &= -\Gamma \tilde \Gamma^{-1}(a_r + p_r h_1(t))v_{1,1}(t) + R_1 \tilde v(t) \\
& \qquad + \sum_{i=1}^{r} R_i \big( y^{(i-1)}(t)-\tilde v^{(i-1)}(t) \big) 
+ f\big( d(t),T(y,\dot y,\ldots,y^{(r-1)})(t) \big). 
\end{aligned}
\end{equation*}
We record some useful observations
\begin{equation*}
\begin{aligned}
y(t) - \tilde v(t) &= y(t) - \sum_{i=1}^{r-1} v_{i,1}(t) \\
&= y(t) - (y(t) - z_{2,1}(t)) - \ldots - (z_{r-2,1}(t) - z_{r-1,1}(t)) \\
& = z_{r-1,1}(t) = z(t),
\end{aligned}
\end{equation*}
the following relation for $z_{1,r}$
\begin{equation*}
z_{1,r}(t) = z_{1,1}^{(r-1)}(t) - \sum_{k=0}^{r-2} \left( \frac{\text{\normalfont{d}}}{\text{\normalfont{d}} t} \right)^k \Big[ ( a_{r-k-1} + p_{r-k-1} h_{1}(t)) v_{1,1}(t) \Big],
\end{equation*}
and
\begin{equation*}
z_{1,1}(t) = y(t) - v_{1,1}(t) = z(t) + \tilde v(t) - v_{1,1}(t) = z(t) + \sum_{i=2}^{r-1} v_{i,1}(t).
\end{equation*}
Therefore, we have
\begin{equation*}
z_{1,r}(t) = z^{(r-1)}(t) + \sum_{i=2}^{r-1} v_{i,1}^{(r-1)}(t) 
-  \sum_{k=0}^{r-2} \left( \frac{\text{\normalfont{d}}}{\text{\normalfont{d}} t} \right)^k \Big[ ( a_{r-k-1} + p_{r-k-1} h_{1}(t)) v_{1,1}(t) \Big].
\end{equation*}
Now, for $G: = I - \Gamma \tilde \Gamma^{-1}$ and $j=1,\ldots,r-1$ we define $w_{i,j}(\cdot) := v_{i,j}(\cdot)$ for $i=2,\ldots,r-1$ and $j=1,\ldots,r$, 
and $w_{1,r}(\cdot) := v_{1,r}(\cdot)$. Further, for $j=1,\ldots,r$ we set
\begin{equation*}
\begin{aligned}
w_{1,r-j}(t) := v_{1,r-j}(t) &+ G  \sum_{k=2}^{r-1} v_{k,1}^{(r-1-j)}(t) \\
& -G \sum_{k=j}^{r-2} \left( \frac{\text{\normalfont{d}}}{\text{\normalfont{d}} t} \right)^{k-j} \left[ ( a_{r-k-1} + p_{r-k-1} h_{1}(t)) v_{1,1}(t) \right] .
\end{aligned}
\end{equation*}
Next, we investigate the dynamics of $w_{i,j}$. We set $\tilde w(\cdot) := \sum_{i=2}^{r-1} w_{i,1}(\cdot)$, and obtain, for $i=1$
\begin{subequations} \label{eq:ddt_w}
\begin{equation} \label{eq:ddt_w1}
\begin{aligned}
\dot w_{1,1}(t) &= w_{1,2}(t) - \Gamma \tilde \Gamma^{-1} ( a_1 + p_1 h_1(t)) (w_{1,1}(t) - G \tilde w(t)) \\
& \qquad + R_{r} (w_{1,1}(t) + \Gamma  \tilde \Gamma^{-1}\tilde w(t) ) , \\
\dot w_{1,2}(t) &= w_{1,3}(t) - \Gamma \tilde \Gamma^{-1} ( a_2 + p_2 h_1(t)) (w_{1,1}(t) - G \tilde w(t))\\
& \qquad + R_{r-1}(w_{1,1}(t) + \Gamma  \tilde \Gamma^{-1}\tilde w(t) ) , \\
 & \ \, \vdots \\
\dot w_{1,r-2}(t) &= w_{1,r-1}(t) - \Gamma \tilde \Gamma^{-1} ( a_{r-2} + p_{r-2} h_1(t)) (w_{1,1}(t) - G \tilde w(t)) \\
& \qquad + R_{3} (w_{1,1}(t) + \Gamma  \tilde \Gamma^{-1}\tilde w(t) ) , \\
\dot w_{1,r-1}(t) &= w_{1,r}(t) - \Gamma \tilde \Gamma^{-1} ( a_{r-1} + p_{r-1} h_1(t)) (w_{1,1}(t) - G \tilde w(t)) \\
& \qquad + R_{2} (w_{1,1}(t) + \Gamma  \tilde \Gamma^{-1}\tilde w(t) ) - G z^{(r-1)}(t) , \\
\dot w_{1,r}(t) &=   -\Gamma \tilde \Gamma^{-1} ( a_{r} + p_{r} h_1(t)) (w_{1,1}(t) - G \tilde w(t)) + R_{1} (w_{1,1}(t) + \Gamma  \tilde \Gamma^{-1}\tilde w(t) ) \\
 & \qquad + \sum_{i=1}^{r}R_i z^{(i-1)}(t) + f\big(d(t),T(y,\dot y,\ldots,y^{(r-1)})(t) \big) ,
\ \\
h_1(t) &= \frac{1}{1-\vp_1(t)^2 \|w_{1,1}(t) - G \tilde w(t) \|^2} \\
\end{aligned}
%
%
\end{equation}
for $i=2$
\begin{equation} \label{eq:ddt_w2}
\begin{aligned}
\dot w_{2,1}(t) &= w_{2,2}(t) - (a_1 + p_1 h_2(t)) w_{2,1}(t) + (a_1 + p_1 h_{1}(t))(w_{1,1}(t)-G\tilde w(t)) , \\
 & \ \, \vdots \\
\dot w_{2,r-1}(t) &= w_{2,r}(t) - (a_{r-1} + p_{r-1} h_2(t)) w_{2,1}(t) \\
& \qquad  + (a_{r-1} + p_{r-1} h_{1}(t))(w_{1,1}(t)-G\tilde w(t)) , \\
\dot w_{2,r}(t) &= \qquad  - (a_{r} + p_{r} h_2(t)) w_{2,1}(t) + (a_{r} + p_{r} h_{1}(t))(w_{1,1}(t)-G\tilde w(t)) , \\
\ \\
h_2(t) &= \frac{1}{1-\vp(t)^2 \|w_{2,1}(t) \|^2} , \\
\end{aligned}
\end{equation}
and for $i=3,\ldots,r-1$ we find
\begin{equation} \label{eq:ddt_wi}
\begin{aligned}
\dot w_{i,1}(t) &= w_{i,2}(t) - (a_1 + p_1 h_i(t)) w_{i,1}(t) + (a_1 + p_1 h_{i-1}(t)) w_{i-1,1}(t) , \\
 & \ \, \vdots \\
\dot w_{i,r-1}(t) &= w_{i,r}(t) - (a_{r-1} + p_{r-1} h_i(t)) w_{i,1}(t) + (a_{r-1} + p_{r-1} h_{i-1}(t)) w_{i-1,1}(t) , \\
\dot w_{i,r}(t) &= \qquad \quad \ - (a_{r \quad} + p_{r \quad} h_i(t)) w_{i,1}(t) + (a_{r \quad } + p_{r \quad } h_{i-1}(t))w_{i-1,1}(t) , \\
\ \\
h_i(t) &= \frac{1}{1-\vp(t)^2 \|w_{i,1}(t) \|^2} .
\end{aligned}
\end{equation}
\end{subequations}
\underline{Step 2}: 
For~$\bar q = rm(r-1) + r$ we define the operator $\tilde T : \cC([-\tau,\infty) \to \R^{rm}) \to \cL_{\rm loc}^\infty(\rp \to \R^{\bar q})$ as the solution operator of~\eqref{eq:ddt_w} in the following sense: 
for $\xi_1,\ldots,\xi_r \in \cC([0,\infty) \to \R^m)$ let $w_{i,j} : [0,\omega) \to \R^m$, $\omega \in (0,\infty]$ be the unique maximal solution of~\eqref{eq:ddt_w},
with $z= \xi_1, \dot z = \xi_2, \ldots, z^{(r-1)} = \xi_r $, and with suitable initial values~$w_{i,j}(0)$ according to the transformations. 
Then we define for $t \in [0,\omega)$
\begin{equation*}
\begin{small}
\tilde T(\xi_1,\ldots,\xi_r)(t) := \big( w_{1,1}(t),\ldots,w_{1,r}(t),w_{2,1}(t),\ldots,w_{r-1,r}(t),h_1(t),\ldots,h_{r-1}(t) \big)^\top. 
\end{small}
\end{equation*}
Note, that in~\eqref{eq:ddt_w1} the terms $y,\dot y,\ldots, y^{(r-1)}$ can be expressed in terms of $w_{i,j}$ and $z,\dot z,\ldots, z^{(r-1)}$ 
using $y^{(i)} = z^{(i)} + w_{1,1}^{(i)} + \Gamma \tilde \Gamma^{-1} \tilde w^{(i)}$ and the equations~\eqref{eq:ddt_w}.
Now, for
\begin{equation*}
\begin{aligned}
\cD := 
\setdef{ 
(t,\zeta_{1,1},\ldots,\zeta_{r-1,r})  \in \rp \times \R^{rm(r-1)}   }
{ \begin{array}{l}
 \vp_1(t) \| \zeta_{1,1}(t) - G \tilde  \zeta(t) \| < 1, \\
\vp(t) \| \zeta_{i,1}(t) \| < 1, \\ 
i=2,\ldots,r-1 
\end{array}
},
\end{aligned}
\end{equation*}
where $\tilde \zeta(\cdot) := \sum_{i=2}^{r-1}\zeta_{i,1}(\cdot)$, we have $(t,w_{1,1}(t),\ldots,w_{r-1,r}(t)) \in \cD$ for all $t \in [0,\omega)$. 
Furthermore, the closure of the graph of the solution $(w_{1,1},\ldots,w_{r-1,r})$ of~\eqref{eq:ddt_w} is not a compact subset of~$\cD$.

\ \\
Next, we show~$\tilde T \in \cT^{rm,\bar q}_{\tau}$; first we show that property~\ref{T:BIBO} from Definition~\ref{Def:OP-T} is satisfied.
To this end, we assume that $z,\dot z,\ldots, z^{(r-1)}$ are bounded on $[0,\omega)$. 
As the solution evolves in~$\cD$, $w_{1,1}- G \tilde w, w_{2,1},\ldots,w_{r-1,1}$ are bounded. Thus, $y = z + w_{1,1} + \Gamma \tilde \Gamma^{-1} \tilde w$ is bounded and hence $T\left(y,\dot y,\ldots, y^{(r-1)}\right)$ is bounded via~$T \in \cT^{rm,q}_{\tau,1}$, and therefore $f\left(d(\cdot),T\left(y,\dot y,\ldots, y^{(r-1)}\right)(\cdot)\right)$ is bounded on~$[0,\omega)$.

\ \\
\underline{Step 2a}: 
We show $w_{i,j} \in \cL^\infty([0,\omega) \to \R^m)$ for $i=1,\ldots,r-1$ and $j=1,\ldots,r$.
We set $w_i := (w_{i,1}^\top,\ldots,w_{i,r}^\top)^\top \in \R^{rm}$ for $i=1,\ldots,r-1$ and~$\bar w := w_{1,1} - G \tilde w$. 
For corresponding matrices~$A,Q,P$ from satisfying~\ref{Ass:A-Hurwitz}, using the Kronecker matrix product, we define
\begin{equation} \label{def:hatAPQ}
\hat A := A \otimes I_m \in \R^{rm \times rm}, \quad \hat P := P \otimes I_m \in \R^{rm \times rm}, \quad \hat Q := Q \otimes I_m \in \R^{rm \times rm}.
\end{equation}
Then, using~\cite[Fact 7.4.34]{Bern09}, we have $\sigma(\hat A) = \sigma(A)$, $\sigma(\hat P) = \sigma(P)$ and $\sigma(\hat Q) = \sigma(Q)$, and
\begin{equation} \label{eq:hatAhatP+hatPhatA}
\hat A^\top \hat P + \hat P \hat A + \hat Q = 0.
\end{equation}
Furthermore, for $p_1,\ldots,p_r$ from~\ref{Ass:A-Hurwitz}, setting $\bar P := (p_1,\ldots,p_{r})^\top \otimes I_m$ we have
\begin{equation}  \label{eq:hatPbarP}
\hat P \bar P = \big[ \tilde p I_m , 0 ,\ldots, 0 \big]^\top \in \R^{rm \times m},
\end{equation}
where $\tilde p := P_1 - P_2 P_4^{-1} P_2^\top > 0$.
Then we may rewrite~\eqref{eq:ddt_w} as
\begin{equation} \label{eq:ddt_w-overall}
\begin{aligned}
\dot w_1(t) &= \hat A w_1(t) - h_1(t) \bar P \Gamma \tilde \Gamma^{-1} \bar w(t) + B_1(t) , \\
\dot w_2(t) &= \hat A w_2(t) - h_2(t) \bar P w_{2,1}(t) + h_1(t) \bar P \bar w(t) + B_2(t)  , \\
\dot w_i(t) &= \hat A w_i(t) - h_i(t) \bar P w_{i,1}(t) + h_{i-1}(t) \bar P w_{i-1,1}(t) + B_i(t) , 
\end{aligned}
\end{equation}
for $i=3,\ldots,r-1$ and suitable bounded functions $B_1, B_2, B_i \in \cL^\infty([0,\omega) \to \R^{rm})$, respectively.

\ \\
Seeking a suitable Lyapunov function for the overall system~\eqref{eq:ddt_w-overall} we define
with~$ \hat P$ from~\eqref{def:hatAPQ} 
\begin{equation*}
\hat P_1 := \left(I_r \otimes (\Gamma \tilde \Gamma^{-1})^{-\frac{1}{2}}  \right) \hat P \left(I_r \otimes (\Gamma \tilde \Gamma^{-1})^{-\frac{1}{2}}  \right),
\end{equation*}
which is possible since~\ref{Ass:Gam-tildeGam} is satisfied by assumption. 
Observe $P_1 = P_1^\top > 0$ and, recalling $\hat P \bar P = [\tilde p I_m,0,\ldots,0]^\top \in \R^{rm \times m}$ via~\eqref{eq:hatPbarP}, 
we have 
\begin{equation} \label{eq:P1-barP}
\hat P_1  \bar P = [\tilde p (\Gamma \tilde \Gamma^{-1})^{-1},0,\ldots,0]^\top \in \R^{rm \times m}.
\end{equation}
Since for all $M \in \R^{m \times m}$ we have $\hat A (I_r \otimes M) = (I_r \otimes M) \hat A$,
$\hat A^\top (I_r \otimes M) = (I_r \otimes M) \hat A^\top $ 
we obtain
\[
\hat A^\top \hat P_1 = \left(I_r \otimes (\Gamma \tilde \Gamma^{-1})^{-\frac{1}{2}}  \right) \hat A^\top \hat P \left(I_r \otimes (\Gamma \tilde \Gamma^{-1})^{-\frac{1}{2}}\right)
\]
and respective for $\hat P_1 \hat A$. Therefore, 
\begin{equation} \label{eq:AP1-P1A-Q1}
\begin{aligned}
\hat A^\top \hat P_1 + \hat P_1 \hat A &= \left(I_r \otimes (\Gamma \tilde \Gamma^{-1})^{-\frac{1}{2}}  \right) \hat A^\top \hat P + \hat P \hat A \left(I_r \otimes (\Gamma \tilde \Gamma^{-1})^{-\frac{1}{2}}\right)  \\
&\overset{\eqref{eq:hatAhatP+hatPhatA}}{ =} - \left(I_r \otimes (\Gamma \tilde \Gamma^{-1})^{-\frac{1}{2}}  \right) \hat Q \left(I_r \otimes (\Gamma \tilde \Gamma^{-1})^{-\frac{1}{2}}\right) 
 =: - \hat Q_1,
\end{aligned}
\end{equation}
where $\hat Q_1 = \hat Q_1^\top$ by~\ref{Ass:Gam-tildeGam}, and $\hat Q_1 > 0$ via~\eqref{eq:hatAhatP+hatPhatA} 
and~\ref{Ass:Gam-tildeGam}.
We define
\begin{equation*}
0 < \cP := \begin{bmatrix} \hat P_1 & 0  &  \ldots & 0 \\
0 & \hat P &  &  \vdots \\
\vdots & & \ddots & 0 \\
0 & \ldots & 0 &\hat P \end{bmatrix} = \cP^\top \in \R^{rm(r-1) \times rm(r-1)},
\end{equation*}
and set $w := (w_1^\top,\ldots,w_{r-1}^\top)^\top \in \R^{m r(r-1)}$.

\ \\
Now, we consider the Lyapunov function candidate
\begin{equation*}
\begin{aligned}
V: \R^{rm} \times \cdots \times \R^{rm} &\to \R, \\
(w_1,\ldots,w_{r-1}) &\mapsto w^\top \cP w = w_1^\top \hat P_1 w_1 + \sum_{i=2}^{r-1} w_i^\top \hat P w_i , \\
\end{aligned}
\end{equation*}
and study its evolution along the solution trajectories of the respective differential equations~\eqref{eq:ddt_w-overall}. 
We fix~$\theta \in (0,\omega)$ and note that~$w_i \in \cL^\infty([0,\theta) \to \R^{rm})$ for all $i=1,\ldots,r-1$. 
Using~\eqref{eq:P1-barP} and~\eqref{eq:AP1-P1A-Q1} we obtain for~$t \in [\theta, \omega)$
\begin{equation} \label{eq:Lyapunov}
\begin{aligned}
\ddt w(t)^\top \cP w(t) =& \
w_1(t)^\top (\hat A^\top \hat P_1 + \hat P_1 \hat A)  w_1(t) \\
& - 2 h_1(t) w_{1}(t)^\top \hat P_1 \bar P \Gamma \tilde \Gamma^{-1} \bar w(t) + 2 w_1(t)^\top \hat P_1 B_1(t) \\
&+ w_2(t)^\top (\hat A^\top \hat P + \hat P \hat A) w_2(t) - 2 h_2(t) w_2(t)^\top \hat P \bar P w_{2,1}(t) \\
& + 2 h_1(t) w_{2}(t)^\top \hat P \bar P \bar w(t) + 2 w_2(t)^\top \hat P B_2(t) \\
& + \sum_{i=3}^{r-1} \Bigg[ w_i(t)^\top (\hat A^\top \hat P + \hat P \hat A) w_i(t) - 2 h_i(t) w_i(t)^\top \hat P \bar P w_{i,1}(t) \\
& \qquad + 2 h_{i-1}(t) w_i(t)^\top \hat P \bar P w_{i-1,1}(t) + 2 w_i(t)^\top \hat P B_i(t) \Bigg] \\
\le & - \lambda_{\rm min}(\hat Q_1) \|w_1(t)\|^2 - 2 \tilde p h_{1}(t) w_{1,1}(t)^\top \bar w(t)  + 2 \|P_1\| \|B_1\|_\infty \|w_1(t) \| \\
& - \lambda_{\rm min}(\hat Q) \|w_{2}(t)\|^2 - 2\tilde p h_2(t) \|w_{2,1}(t)\|^2 + 2 \tilde p h_1(t) w_{2,1}(t)^\top \bar w(t) \\
&+ 2 \|\hat P\| \|B_2\|_\infty \|w_2(t)\|  \\
& + \sum_{i=3}^{r-1} \Bigg[- \lambda_{\rm min}(\hat Q) \|w_i(t)\|^2  - 2\tilde p h_i(t) \|w_{i,1}(t)\|^2 \\
& \qquad + 2 \tilde p h_{i-1}(t) w_{i,1}(t)^\top w_{i-1,1}(t) + 2 \|\hat P\| \|B_i\|_\infty \|w_i(t)\| \Bigg] \\
\le &- \lambda_{\rm min}(\hat Q_1) \|w_1(t)\|^2 + 2 \|P_1\| \|B_1\|_\infty \|w_1(t) \| \\
& + \sum_{i=2}^{r-1} \Big[ - \lambda_{\rm min}(\hat Q)\|w_i(t)\|^2 + 2 \|\hat P\| \|B_i\|_\infty \|w_i(t)\| \Big] \\
& - 2 \tilde p h_{1}(t) w_{1,1}(t)^\top \bar w(t) \\
& \red{{- 2 \tilde p h_2(t) \|w_{2,1}(t)\|^2 }} + 2 \tilde p h_1(t) w_{2,1}(t)^\top \bar w(t) \\
& \blue{{- 2 \tilde p h_3(t) \|w_{3,1}(t)\|^2}} \red{{+ 2 \tilde p h_2(t) w_{3,1}(t)^\top w_{2,1}(t)}} \\
& - 2 \tilde p h_4(t) \|w_{4,1}(t)\|^2  \blue{{+ 2 \tilde p h_3(t) w_{4,1}(t)^\top w_{3,1}(t)}} \\
& \ \, \vdots \\
& \green{{- 2 \tilde p h_{r-2}(t) \|w_{r-2,1}(t)\|^2}} + 2 \tilde p h_{r-3}(t) w_{r-2,1}(t)^\top w_{r-3,1}(t)  \\
& - 2 \tilde p h_{r-1}(t) \|w_{r-1,1}(t)\|^2 \green{{+ 2 \tilde p h_{r-2}(t) w_{r-1,1}(t)^\top w_{r-2,1}(t) }} ,
%
%
%
\end{aligned}
\end{equation}
where $\lambda_{\rm min}(\hat Q), \lambda_{\rm min}(\hat Q_1)$ denotes the smallest eigenvalue of~$\hat Q$ and~$\hat Q_1$, respectively.

\ \\
In order to proceed with the estimation of each term in~\eqref{eq:Lyapunov}, we record the following observation. 
Due to~\ref{Ass:funnel-functions} we have for $i=2,\ldots,r-2$ and~$t \in [\theta,\omega)$
\begin{equation} \label{eq:estimate_hi-wi1}
\begin{aligned}
& - h_{i}(t) \|w_{i,1}(t)\|^2 + h_i(t)w_{i,1}(t)^\top w_{i+1,1}(t) \\
 \le & - h_{i}(t) \|w_{i,1}(t)\|^2 + h_i(t) \|w_{i,1}(t)\| \|w_{i+1,1}(t)\| \\
 < &- h_{i}(t) \|w_{i,1}(t)\|^2 + h_i(t) \|w_{i,1}(t)\| \frac{1}{\vp(t)} \\
 = &- h_{i}(t) \|w_{i,1}(t)\| \left( \|w_{i,1}(t)\| - \frac{1}{\vp(t)} \right) \\
= &- \| w_{i,1}(t)\| \frac{1}{(1+\vp(t)\| w_{i,1}(t)\| )(1-\vp(t) \| w_{i,1}(t)\|)} \left( \|w_{i,1}(t)\| - \frac{1}{\vp(t)} \right) \\
= &-\| w_{i,1}(t)\| \frac{1}{(1+\vp(t)\| w_{i,1}(t)\| )(1-\vp(t) \| w_{i,1}(t)\|)} \left( \vp(t) \|w_{i,1}(t)\| - 1 \right) \frac{1}{\vp(t)} \\
= & \ \| w_{i,1}(t)\| \frac{1}{1+\vp(t)\| w_{i,1}(t)\| } \frac{1}{\vp(t)}\\
 \le & \  \| w_{i,1}(t)\| \frac{1}{\vp(t)} \le \| w_{i,1}(t)\| \sup_{s \ge \theta} \frac{1}{\vp(s)} 
 \le \| w_{i}(t)\| \sup_{s \ge \theta} \frac{1}{\vp(s)}.
\end{aligned}
\end{equation}
We recall $\bar w(\cdot) = w_{1,1}(\cdot) - G\tilde w(\cdot)$ and observe $\|w_{2,1}(t)\| < \vp(t)^{-1}$ and $\| \tilde w(t) \| < \sum_{i=2}^{r-1}\vp(t)^{-1} = (r-2) \vp(t)^{-1}$ 
for~$t \in [\theta,\omega)$. 
Therefore, we obtain for~$t \in [\theta, \omega)$
\begin{equation} \label{eq:estimate_annoying-term}
\begin{aligned}
& -  h_1(t) w_{1,1}(t)^\top \bar w(t) +  h_1(t) w_{2,1}(t)^\top \bar w(t) \\
= & -h_1(t) \bar w(t)^\top \bar w(t) - h_1(t) (G \tilde w(t))^\top \bar w(t) +  h_1(t) w_{2,1}(t)^\top \bar w(t) \\
\le &-  h_1(t) \|\bar w(t)\|^2 +  h_1(t) \big( \|G\| \|\tilde w(t)\| + \|w_{2,1}(t)\| \big) \|\bar w(t)\| \\
\le & - h_1(t) \|\bar w(t)\| \left( \|\bar w(t)\| - \frac{\|G\|(r-2) + 1}{\vp(t)} \right) \\
\overset{\rm \ref{Ass:G}}{<} & - h_1(t) \|\bar w(t)\| \left( \|\bar w(t)\| - \frac{\frac{\rho - 1}{r-2} (r-2) + 1}{\vp(t)} \right) \\
\overset{\rm \ref{Ass:funnel-functions}}{=} & - h_1(t) \|\bar w(t)\| \left( \|\bar w(t)\| - \frac{1}{\vp_1(t)} \right) \\
\le & \  \|\bar w(t)\| \sup_{s \ge \theta} \frac{1}{\vp_1(s)} \\
\le &  \sup_{s \ge \theta} \frac{1}{\vp_1(s)}\big( \|w_{1,1}(t) - G \tilde w(t)\| \big) \\
\le &  \sup_{s \ge \theta} \frac{1}{\vp_1(s)}\big( \|w_{1,1}(t)\| + \|G\| \|\tilde w(t)\| \big) \\
\le &  \sup_{s \ge \theta} \frac{1}{\vp_1(s)}\left( \|w_{1}(t)\| + \|G\| \sum_{i=2}^{r-1}\|w_i(t)\| \right).
\end{aligned}
\end{equation}
We set
\begin{equation*}
\begin{aligned}
M_1 &:=  \| \hat P_1\|\|B_1\|_\infty + \tilde p \sup_{s \ge \theta} \vp_1(s)^{-1}, \\
M_i & := \|\hat P\|\|B_i\|_\infty +  \tilde p \sup_{s \ge \theta} \vp(s)^{-1} + \tilde p \|G\| \sup_{s \ge \theta} \vp_1(s)^{-1}, \ i=2,\ldots,r-1, \\
N & := \tfrac{2 M_1^2}{\lambda_{\rm min}(\hat Q_1)} + \sum_{i=2}^{r-1}\tfrac{2 M_i^2}{\lambda_{\rm min}(\hat Q)} .
\end{aligned}
\end{equation*}
Then, using $2ab \le 2a^2 + \tfrac{1}{2}b^2$ for $a,b \in \R$, 
we may estimate~\eqref{eq:Lyapunov} with the aid of~\eqref{eq:estimate_hi-wi1} and~\eqref{eq:estimate_annoying-term} 
for~$t \in [\theta,\omega)$ 
\begin{equation*}
\begin{aligned}
\ddt w(t)^\top \cP w(t) \le & 
- \lambda_{\rm min}(\hat Q_1) \|w_1(t)\|^2   
- \lambda_{\rm min}(\hat Q) \sum_{i=2}^{r-1} \|w_i(t)\|^2 \\
&  + 2 \left(\| \hat P_1\| \|B_1\|_\infty+\tilde p \sup_{s \ge \theta} \vp_1(s)^{-1} \right) \|w_1(t) \|  \\
&  + \sum_{i=2}^{r-1} 2 \left(\|\hat P\| \|B_i\|_\infty + \tilde p \sup_{s \ge \theta} \vp_i(s)^{-1} + \tilde p \|G\| \sup_{s \ge \theta} \vp_1(s)^{-1} \right) \|w_i(t)\| \\
\le & -  \lambda_{\rm min}(\hat Q_1) \|w_1(t)\|^2   - \lambda_{\rm min}(\hat Q) \sum_{i=2}^{r-1} \|w_i(t)\|^2 + \sum_{i=1}^{r-1} 2 M_i \|w_i(t)\|  \\
=& -  \lambda_{\rm min}(\hat Q_1) \|w_1(t)\|^2   - \lambda_{\rm min}(\hat Q) \sum_{i=2}^{r-1} \|w_i(t)\|^2 \\
& + \frac{ 2 M_1}{\sqrt{\lambda_{\rm min}(\hat Q_1)}} \|w_1(t)\| \sqrt{\lambda_{\rm min}(\hat Q_1)} \\
& + \sum_{i=2}^{r-1} \frac{2 M_i}{\sqrt{\lambda_{\rm min}(\hat Q)}} \|w_i(t)\| \sqrt{\lambda_{\rm min}(\hat Q)} \\
\le & -  \frac{\lambda_{\rm min}(\hat Q_1)}{2} \|w_1(t)\|^2   - \frac{\lambda_{\rm min}(\hat Q)}{2} \sum_{i=2}^{r-1} \|w_i(t)\|^2 + N \\
\le & - \frac{\mu}{2} w(t)^\top \cP w(t) + N ,
\end{aligned}
\end{equation*}
where $\mu := \tfrac{\min\{\lambda_{\rm min}(\hat Q), \lambda_{\rm min}(\hat Q_1)\}}{\lambda_{\rm max}(\cP)} =  \tfrac{\min\{\lambda_{\rm min}(\hat Q), \lambda_{\rm min}(\hat Q_1)\}}{\max\{\lambda_{\rm max}(\hat P),\lambda_{\rm max}(P_1)\}}> 0$. 
With the aid of Gr\"onwall's lemma we obtain 
\begin{equation*}
w(t)^\top \cP w(t) \le w(\theta)^\top \cP w(\theta) e^{-\frac{\mu}{2} (t-\theta)} + \frac{2N}{\mu} 
\end{equation*}
and therefore,
\begin{equation} \label{eq:est-w}
\|w(t)\|^2 \le \frac{\lambda_{\rm max}(\cP)}{\lambda_{\rm min}(\cP)} \|w(\theta)\|^2 e^{-\frac{\mu}{2} (t-\theta)} + \frac{2N}{\mu \lambda_{\rm min}(\cP)}.
\end{equation}
Inequality~\eqref{eq:est-w} implies $w \in \cL^\infty\left([\theta,\omega) \to \R^{rm(r-1)}\right)$ and hence we have
$w_i \in \cL^\infty([0,\omega) \to \R^{rm})$ for all $i=1,\ldots,r-1$. 
In particular, we obtain $\bar w \in \cL^\infty([0,\omega) \to \R^m)$. 

\ \\
\underline{Step 2b}: 
We show $h_i \in \cL^\infty([0,\omega) \to \R)$, for all $i=1,\ldots,r-1$.
For the sake of better legibility we set $x_1(\cdot) := \bar w(\cdot)$ and $x_i(\cdot) := w_{i,1}(\cdot)$ for $i=2,\ldots,r-1$.
Observing $- \Gamma \tilde \Gamma^{-1} - G = -\Gamma \tilde \Gamma^{-1} - (I- \Gamma \tilde \Gamma^{-1}) = -I_m$, 
and setting $\tilde w_2(\cdot) := \sum_{i=2}^{r-1} w_{i,2}(\cdot)$ we obtain via~\eqref{eq:ddt_w} 
\begin{equation} \label{eq:ddt_x1}
\begin{aligned}
 & \dot x_1(t) = \dot w_{1,1}(t) - G \ddt \tilde w(t) \\
&=   w_{1,2}(t) - \Gamma \tilde \Gamma^{-1} ( a_1 + p_1 h_1(t)) (w_{1,1}(t) - G \tilde w(t)) + R_{r} (w_{1,1}(t) + \Gamma  \tilde \Gamma^{-1}\tilde w(t) ) \\
& \quad - G \Big[ \tilde w_{2}(t) - (a_1 + p_1 h_{r-1}(t))w_{r-1,1}(t) + (a_{1} + p_{1} h_{1}(t))(w_{1,1}(t)-G\tilde w(t))  \Big] \\
& =  w_{1,2}(t) - G \tilde w_2(t) + G(a_1 + p_1 h_{r-1}(t))w_{r-1,1}(t) + R_{r} (w_{1,1}(t) + \Gamma  \tilde \Gamma^{-1}\tilde w(t) ) \\
& \quad -  \Gamma \tilde \Gamma^{-1} ( a_1 + p_1 h_1(t)) (w_{1,1}(t) - G \tilde w(t)) - G ( a_1 + p_1 h_1(t)) (w_{1,1}(t) - G \tilde w(t)) \\
& =  w_{1,2}(t) - G \tilde w_2(t) + \underbrace{(-\Gamma \tilde \Gamma^{-1} - G)}_{= -I_m}( a_1 + p_1 h_1(t)) \underbrace{(w_{1,1}(t) - G \tilde w(t))}_{= x_1(t)}  \\
& \quad + G(a_1 + p_1 h_{r-1}(t))w_{r-1,1}(t) + R_{r} (w_{1,1}(t) + \Gamma  \tilde \Gamma^{-1}\tilde w(t) ) \\
& = - (a_1 + p_1 h_1(t)) x_1(t) + G(a_1 + p_1 h_{r-1}(t))w_{r-1,1}(t)\\
& \qquad + R_{r} (w_{1,1}(t) + \Gamma  \tilde \Gamma^{-1}\tilde w(t) ) +  w_{1,2}(t) - G \tilde w_2(t). 
\end{aligned}
\end{equation} 
Since~$p_1 =1$, using~\eqref{eq:ddt_w} and~\eqref{eq:ddt_x1} we have for $t \in [0,\omega)$ and $i=2,\ldots,r-1$
\begin{equation} \label{eq:ddt-w-bounded}
\begin{aligned}
\ddt \tfrac{1}{2} \| x_1(t)\|^2 &= -h_1(t) \| x_1(t)\|^2 + h_{r-1}(t) x_1(t)^\top G x_{r-1}(t) + x_1(t)^\top b_1(t), \\
\ddt \tfrac{1}{2} \| x_{i}(t)\|^2 &= -h_i(t) \| x_{i}(t)\|^2 + h_{i-1}(t) x_{i}(t)^\top x_{i-1}(t) + x_{i}(t)^\top b_i(t), 
\end{aligned}
\end{equation}
for suitable functions~$b_i \in \cL^\infty([0,\omega) \to \R^m)$, $i=1,\ldots,r-1$, respectively.
We observe that in~\eqref{eq:ddt-w-bounded} for all $i=2,\ldots,r-1$ the~$i^{\rm th}$ differential equation depends on the preceding gain function~$h_{i-1}$, respectively, and the first differential equation depends on the last gain function~$h_{r-1}$.
Therefore, we cannot apply standard funnel argumentation to show boundedness of the gain functions, cf.\cite[p. 484-485]{IlchRyan02b}, 
\cite[p. 241-242]{Ilch13} or~\cite[p. 350-351]{BergLe18a}.
However, on closer examination of the respective proofs in the aforementioned references we may retain the following. 
Due to the shape of the gain function~$h_i$, boundedness of~$h_i$ on~$[0,\omega)$ is equivalent to the existence of ${\nu_i > 0}$ such that for all~$t \in [0,\omega)$ 
we have~${\vp(t)^{-1} - \|x_i(t)\| \ge \nu_i}$, $i=2,\ldots,r-1$ (${\vp_1(t)^{-1} - \|x_1(t)\| \ge \nu_1}$), respectively.
Moreover, via the loop structure in~\eqref{eq:ddt-w-bounded} it suffices to show boundedness of one gain function, which implies boundedness of all remaining.

\ \\
We define $\psi_1(\cdot) := \vp_1(\cdot)^{-1}$, $\psi(\cdot) := \vp(\cdot)^{-1}$,
set $\lambda_1 := \inf_{s \in (0,\omega)} \psi_1(s) > 0$ and $\lambda := \inf_{s \in (0,\omega)} \psi(s) > 0$, and
fix $\beta \in (0,\omega)$.
Since $\liminf_{t \to \infty} \vp(t) > 0$ and $\dot \vp(\cdot)$ is bounded, we have that $\ddt \psi|_{[\beta, \infty)}(\cdot)$ is bounded, and respective for $\ddt \psi_1|_{[\beta, \infty)}(\cdot)$.
Thus, there exists a Lipschitz bound $L > 0$ of $\psi|_{[\beta,\infty)}(\cdot)$, and $L_1$ of $\psi_{1}|_{[\beta,\infty)}(\cdot)$.
For~$\rho>1$ as in~\ref{Ass:funnel-functions} we fix~$\delta > 0$ as
\begin{equation} \label{def:delta}
\frac{1}{\rho + 1} < \delta < \frac{1}{2},
\end{equation}
and define
\begin{equation} \label{def:Delta}
\Delta_1 := \rho - \|G\|\left( 4 \rho^2 (\rho+1)^{r-2} - 1 \right) \overset{\rm \ref{Ass:G}}{>} 0, 
\quad \Delta := 1-2 \delta \overset{\eqref{def:delta}}{>} 0.
\end{equation}
With this we choose
\begin{equation*} 
\begin{aligned}
0 < \kappa < \min &\Bigg\{ \frac{\lambda_1}{1+\rho}, 
\frac{\lambda}{2}, 
\inf_{s \in (0,\beta]} \left(\psi_1(s) - \| x_{1}(s)\| \right), \\
& \qquad \min_{i \in \{2,\ldots,r-1\}} \left\{\inf_{s \in (0,\beta]} \left(\psi(s) - \| x_{i}(s)\| \right)\right\} \Bigg\} 
\end{aligned}
\end{equation*}
small enough such that for~$\Delta_1, \, \Delta > 0$ from~\eqref{def:Delta}
\begin{subequations}
\begin{align}
0 < L_1  \le  \min &\Bigg\{ 
\frac{ \lambda_1^2}{4 \kappa} \frac{\rho - \|G\|}{\rho} - \sup_{s \in [0,\omega)}\|b_1(s)\|,  \label{eq:L1} \\
& \qquad
 \frac{\Delta_1 \lambda_1^2}{4 \rho \kappa} - 2 \|G\| \sup_{s \in [\beta,\omega)} \psi_1(s) \rho (\rho+1)^{r-2} - \sup_{s \in [0,\omega)} \|b_1(s)\|
\Bigg\} \nonumber, \\
0 < L  \le \min &\Bigg\{ \frac{\rho^2 \lambda^2}{2 \kappa}  - \sup_{s \in [0,\omega)}  \|b_2(s)\|, \label{eq:L}  \\
& \qquad 
\min_{i \in \{3,\ldots,r-1\}} \left\{ 2^{i-1} \Delta \, \frac{\rho^2 \lambda^2}{ \kappa} - \sup_{s \in [0,\omega) } \| b_i(s)\| \right\} \Bigg\}.
 \nonumber
\end{align}
\end{subequations}
%
Note that~$L_1, L$ are well defined since~$\Delta_1, \, \Delta > 0$ and~$\rho > \|G\|$.
%
Further, since~$\rho > 1$ we choose~$\hat \delta $ such that
\begin{equation} \label{def:hatdelta}
\frac{1}{\rho} < \hat \delta \le 1 .
\end{equation}
With this we define
\begin{equation} \label{eq:kappas}
\begin{aligned}
\kappa_1 & := \kappa, \\
\kappa_i & :=  \frac{\delta^{i-1}}{2\rho^2} \, \kappa, \quad i=2,\ldots,r-2, \\
\kappa_{r-1} & := \frac{ \delta^{r-2} \hat \delta}{2\rho^2} \, \kappa.
\end{aligned}
\end{equation}
Note that these definitions imply $\kappa_{r-1} < \kappa_{r-2} < \cdots < \kappa_2 < \kappa_1$.
We show that for all $t \in [0,\omega)$ we have $\psi_1(t) - \|x_1(t)\| \ge \kappa_1$ 
and $\psi(t) - \|x_i(t)\| \ge \kappa_i$ for all $i=2,\ldots,r-1$;
which is true on~$(0,\beta]$ by definition of~$\kappa$. 
We set 
\begin{equation*}
\begin{aligned}
t_0^1 &:= \inf \setdef{t \in (\beta, \omega)}{ \psi_1(t) - \|x_1(t)\| < \kappa_1}, \\
t_0^i &:= \inf \setdef{t \in (\beta, \omega)}{ \psi(t) - \|x_i(t)\| < \kappa_i}, \  i = 2,\ldots,r-1,
\end{aligned}
\end{equation*}
where the infimum of the empty set is infinity as usual.
Seeking a contradiction we suppose that~$t_0^\ell < \infty$ for some~$\ell \in \{1,\ldots,r-1\}$.
%
%
Then either
\begin{subequations}
\begin{equation} \label{eq:t0:at-least-one-smaller}
\left( \exists\, \ell \in \{2,\ldots,r-1\} :  \ t_0^\ell < t_0^{\ell-1} \right) \ \lor \ t_0^1 < t_0^{r-1},
\end{equation}
or
\begin{equation} \label{eq:t0:all-equal}
t_0^1 = \cdots = t_0^{r-1}.
\end{equation}
\end{subequations}
If~\eqref{eq:t0:all-equal} is true, we decrease~$\hat \delta$ and obtain~$t_0^1 < t_0^{r-1}$; 
therefore without loss of generality we may assume that~\eqref{eq:t0:at-least-one-smaller} is true.
We set~$t_0 := t_0^\ell$, i.e., $t_0$ denotes the moment of the very first complete excess of~$\kappa_\ell$, and distinguish three cases, namely either $\ell = 1$, or $\ell=2$, or $3 \le \ell \le r-1$.

\ \\
If 
$\ell = 1$
there may occur two possible cases, namely either
\begin{subequations}
\begin{equation} \label{eq:t0:psi1-x1-le-psi_r-1-x_r-1}
\psi_1(t_0) - \|x_1(t_0)\| \le \psi(t_0) - \|x_{r-1}(t_0)\|, 
\end{equation}
or
\begin{equation} \label{eq:t0:psi1-x1-g-psi_r-1-x_r-1}
\psi_1(t_0) - \|x_1(t_0)\|  > \psi(t_0) - \|x_{r-1}(t_0)\| .
\end{equation}
\end{subequations}
First, we draw our attention to case~\eqref{eq:t0:psi1-x1-le-psi_r-1-x_r-1} and observe
\begin{equation*}
\begin{aligned}
 \psi_1(t_0) - \|x_1(t_0)\| &\le \psi(t_0) - \|x_{r-1}(t_0)\| \\
\iff  \|x_1(t_0) \| &\ge \psi_1(t_0) -  \psi(t_0) + \|x_{r-1}(t_0)\| \\
& = \frac{\vp(t_0) - \vp_1(t_0)}{\vp_1(t_0) \vp(t_0)} + \|x_{r-1}(t_0)\|,
\end{aligned}
\end{equation*}
and thus, using~\ref{Ass:funnel-functions},
\begin{equation*} 
\begin{aligned}
\|x_1(t_0)\| & \ge \frac{\vp(t_0) - \vp_1(t_0)}{\vp_1(t_0) \vp(t_0)} + \|x_{r-1}(t_0)\|  = \frac{\rho - 1}{\vp(t_0)} + \|x_{r-1}(t_0)\| \\
&> (\rho -1) \| x_{r-1}(t_0)\| + \|x_{r-1}(t_0)\| = \rho \|x_{r-1}(t_0)\|.
\end{aligned}
\end{equation*}
Then, due to the definition of~$t_0$, there exists~$ t_1 \in (t_0, \omega)$ such that $\psi_1( t_1) - \|x_1 ( t_1) \| < \kappa_1$, and 
\begin{equation} \label{eq:x1-bigger-rho-xr-1}
\forall\, t\in [t_0,t_1]: \ \|x_1(t)\| > \rho \|x_{r-1}(t)\|  .
\end{equation}
Thus, we readily deduce the following relations for~$t \in [t_0,  t_1]$
\begin{equation} \label{eq:relations-1-h-x}
\begin{aligned}
\psi_1(t) - \| x_1(t)\| &\le \kappa_1, \\
\| x_1 (t) \| &\ge \psi_1(t) - \kappa_1 \ge \frac{\lambda_1}{2}, \\
h_1(t) = \frac{1}{1-\vp_1(t)^2 \| x_1(t)\|^2} &\ge \frac{\lambda_1}{2 \kappa_1}.
\end{aligned}
\end{equation}
We consider the first equation in~\eqref{eq:ddt-w-bounded} for~$t \in [t_0,  t_1]$
\begin{equation*}
\begin{aligned}
\ddt \tfrac{1}{2} \|x_1(t)\|^2 &= -h_1(t) \|x_1(t)\|^2 + h_{r-1}(t) x_1(t)^\top Gx_{r-1}(t) + x_1(t)^\top b_1(t) \\
&\le -h_1(t) \| x_1(t)\|^2 + h_{r-1}(t) \|x_1(t)\| \|G\| \|x_{r-1}(t)\| + \|x_1(t)\| \|b_1(t)\| \\
& \overset{\eqref{eq:x1-bigger-rho-xr-1}}{<} \big( -h_1(t) \rho +  h_{r-1}(t) \|G\| \big) \frac{\|x_1(t)\|^2}{\rho} + \|x_1(t)\| \|b_1(t)\| ,
\end{aligned}
\end{equation*}
Via condition~\ref{Ass:funnel-functions} and relation~\eqref{eq:x1-bigger-rho-xr-1} we obtain for~$t \in [t_0, t_1]$
\begin{equation*} 
\begin{aligned}
h_1(t) \rho - h_{r-1}(t) \|G\| &= h_1(t) h_{r-1}(t) \Big(\rho - \rho \vp(t)^2 \|x_{r-1}(t)\|^2 \\
& \qquad - \|G\| + \|G\| \vp_1(t)^2 \|x_1(t)\|^2 \Big) \\
& \overset{\eqref{eq:x1-bigger-rho-xr-1}}{>} h_1(t) h_{r-1}(t) \Big( \rho - \rho \vp(t)^2 \|x_{r-1}(t)\|^2 \\
& \qquad - \|G\| + \|G\| \vp_1(t)^2 \rho^2 \|x_{r-1}(t)\|^2 \Big) \\
& \overset{\rm \ref{Ass:funnel-functions}}{ =} h_1(t) h_{r-1}(t)( \rho - \|G\| ) \big( 1- \vp(t)^2 \|x_{r-1}(t)\|^2 \big) \\
& = h_1(t) (\rho - \|G\| ) 
\overset{\eqref{eq:relations-1-h-x}}{\ge} \frac{\lambda_1}{2 \kappa_1} (\rho - \|G\| ) .
\end{aligned}
\end{equation*}
Hence, using the estimation above and the relations from~\eqref{eq:relations-1-h-x} we estimate the first equation in~\eqref{eq:ddt-w-bounded} for~$t  \in [t_0,  t_1]$
\begin{equation*}
\begin{aligned}
\ddt \tfrac{1}{2} \|x_1(t)\|^2 &< - \frac{\lambda_1}{2 \kappa_1} (\rho - \|G\|) \frac{\|x_1(t)\|^2}{\rho} + \|x_1(t)\| \|b_1(t)\| \\
 & \le \left( - \frac{ \lambda_1^2}{4 \kappa_1} \frac{\rho - \|G\|}{\rho} + \sup_{s \in [0,\omega)}\|b_1(s)\| \right) \|x_1(t)\| \\
  & \overset{\eqref{eq:kappas}}{=} \left( - \frac{ \lambda_1^2}{4 \kappa} \frac{\rho - \|G\|}{\rho} + \sup_{s \in [0,\omega)}\|b_1(s)\| \right) \|x_1(t)\| 
\overset{\eqref{eq:L1}}{\le} - L_1 \|x_1(t)\|.
\end{aligned}
\end{equation*}
From this we calculate
\begin{subequations}  \label{eq:contradiction}
\begin{equation}
\begin{aligned}
\|x_1(t_1)\| - \|x_1(t_0)\| &= 
\int_{t_0}^{t_1} \frac{\big( \ddt \tfrac{1}{2}\|x_1(t)\|^2 \big)}{\|x_1(t)\|} \, \dt \\
& \le \int_{t_0}^{t_1} - L_1 \, \dt 
 = -L_1 (t_1 - t_0)
\le \psi_1(t_1) - \psi_1(t_0)
\end{aligned}
\end{equation}
and therefore,
\begin{equation}
\kappa_1 = \psi_1(t_0) - \| x_1(t_0) \| \le \psi_1(t_1) - \|x_1(t_1)\| < \kappa_1,
\end{equation}
\end{subequations}
a contradiction. 
%
Now, we consider the case~\eqref{eq:t0:psi1-x1-g-psi_r-1-x_r-1} where by~$t_0 = t_0^1 < t_0^{r-1}$
\begin{equation*}
 \kappa_{r-1} \le  \psi(t_0) - \|x_{r-1}(t_0)\| < \psi_1(t_0) - \| x_1(t_0) \| = \kappa_1.
\end{equation*}
This, together with the definition of~$t_0$ implies the existence of $t_1 \in (t_0,t_0^{r-1})$ such that
\begin{align} \label{eq:t0:relations_kap_r-1-le-psi-x_r-1-l-psi1-x1-le-kap1}
 & \forall\, t\in [t_0,t_1]\,: \  \kappa_{r-1} \le \psi(t) - \|x_{r-1}(t)\| < \psi_1(t) - \|x_1(t)\| \le \kappa_1
\end{align}
from which we obtain for $t \in [t_0,t_1]$
\begin{subequations} 
\begin{equation} \label{eq:x_r-1-le-x1+kappas}
\begin{aligned}
\|x_{r-1}(t)\| \le \psi(t) -  \kappa_{r-1} & \overset{\rm \ref{Ass:funnel-functions}}{=} \frac{1}{\rho} \left( \frac{1}{\vp_1(t)} - \rho  \kappa_{r-1} \right) \\
&= \frac{1}{\rho} \left( \psi_1(t) - \kappa_1 + \kappa_1 - \rho  \kappa_{r-1} \right) \\
&\le \frac{ \|x_1(t)\| + \kappa_1 - \rho  \kappa_{r-1}}{\rho}
\end{aligned}
\end{equation}
and therefore,
\begin{equation} \label{eq:x1-ge-x_r-1-kappas}
\forall\, t \in [t_0,t_1]\,: \  \|x_1(t)\|  \ge \rho \|x_{r-1}(t)\| - (\kappa_1 - \rho  \kappa_{r-1}) > 0,
\end{equation}
\end{subequations}
the last inequality via
\begin{equation*}
\begin{aligned}
\rho \|x_{r-1}(t)\| - (\kappa_1 - \rho  \kappa_{r-1}) & \overset{\rm \ref{Ass:funnel-functions},~\eqref{eq:t0:relations_kap_r-1-le-psi-x_r-1-l-psi1-x1-le-kap1}}{\ge} 
 \frac{\rho}{\rho \vp_1(t)} - \rho \kappa_1 - (\kappa_1 - \rho  \kappa_{r-1}) \\
&\ge  \lambda_1 - (\rho + 1)\kappa_1 + \rho  \kappa_{r-1} > \rho  \kappa_{r-1}, 
\end{aligned}
\end{equation*}
where we used~$\kappa_1 < \tfrac{\lambda_1}{\rho+1}$ in the last step.
Furthermore, recalling~$h_{r-1}$ and using the definition of~$\kappa_1$ and~$ \kappa_{r-1}$ we deduce for $t \in [t_0,t_1]$
\begin{equation*} 
\begin{aligned}
\frac{1}{h_{r-1}(t)} &= 1- \vp(t)^2 \|x_{r-1}(t)\|^2 = (1+\vp(t)\|x_{r-1}(t)\|)(1-\vp(t)\|x_{r-1}(t)\|) \\
 & = \vp(t) (1+\vp(t) \|x_{r-1}(t)\|) (\psi(t) - \|x_{r-1}(t)\|) \\
 &\ge \vp(t) (1+\vp(t) \|x_{r-1}(t)\|)  \kappa_{r-1} \ge \vp(t)  \kappa_{r-1} = \rho \vp_1(t)  \kappa_{r-1} ,
\end{aligned}
\end{equation*}
and hence for~$t \in [t_0,t_1]$ we have
\begin{equation} \label{eq:h_r-1-le}
h_{r-1}(t) = \frac{1}{1 - \vp(t)^2 \|x_{r-1}(t)\|^2} \le \frac{1}{\vp(t)  \kappa_{r-1}} \overset{\rm \ref{Ass:funnel-functions}}{=} \frac{1}{\rho \vp_1(t)  \kappa_{r-1}}.
\end{equation}
We consider the first equation in~\eqref{eq:ddt-w-bounded} for~$t \in [t_0,  t_1]$ and obtain
\begin{equation*}
\begin{aligned}
\ddt \tfrac{1}{2} \|x_1(t)\|^2 &= -h_1(t) \|x_1(t)\|^2 + h_{r-1}(t) x_1(t)^\top Gx_{r-1}(t) + x_1(t)^\top b_1(t) \\
&\le -h_1(t) \| x_1(t)\|^2 + h_{r-1}(t) \|x_1(t)\| \|G\| \|x_{r-1}(t)\| \\
& \qquad + \|x_1(t)\| \|b_1(t)\| \\
& \overset{\eqref{eq:x_r-1-le-x1+kappas}}{\le} -h_1(t) \| x_1(t)\|^2 + \|G\| h_{r-1}(t) \| x_1(t)\| \frac{\|x_1(t)\| + (\kappa_1 - \rho  \kappa_{r-1})}{\rho}  \\
& \qquad + \|x_1(t)\| \|b_1(t)\| \\
& = \bigg( - \rho h_1(t) + \|G\| h_{r-1}(t) \bigg) \frac{\| x_1(t)\|^2}{\rho} \\
& \qquad + \|G\| \frac{\kappa_1 - \rho  \kappa_{r-1}}{\rho} h_{r-1}(t) \|x_1(t)\| + \|x_1(t)\| \|b_1(t)\| .
\end{aligned}
\end{equation*}
Noting that by definition of~$\kappa_1,\kappa_{r-1}$ we have $\kappa_1 - \rho \kappa_{r-1} = \kappa \, \left(1 - \tfrac{\delta^{r-2} \hat \delta}{2 \rho} \right) > 0$, we estimate for $t \in [t_0,t_1]$ 
\begin{small}
\begin{align*}
\rho  h_1(t)  -  & \|G\| h_{r-1}(t)  \\
&= h_1(t) h_{r-1}(t) \big(\rho - \rho \vp(t)^2 \|x_{r-1}(t)\|^2 - \|G\| + \|G\| \vp_1(t)^2 \|x_1(t)\|^2 \big) \\
& \overset{\eqref{eq:x1-ge-x_r-1-kappas}}{\ge}
h_1(t) h_{r-1}(t) \bigg( \rho - \rho \vp(t)^2 \|x_{r-1}(t)\|^2 \\
&\quad - \|G\| + \|G\| \vp_1(t)^2 \big(\rho \|x_{r-1}(t) - (\kappa_1 - \rho  \kappa_{r-1}) \big)^2 \bigg) \\
&= h_1(t) h_{r-1}(t) \bigg( \rho - \rho \vp(t)^2 \|x_{r-1}(t)\|^2 \\
& \quad - \|G\| + 
\|G\| \vp_1(t)^2 \Big(\rho^2 \|x_{r-1}(t)\|^2 - 2\rho\|x_{r-1}(t)\|(\kappa_1 - \rho  \kappa_{r-1}) + (\kappa_1 - \rho  \kappa_{r-1})^2 \Big) \bigg) \\
& \overset{\rm \ref{Ass:funnel-functions}}{ =} h_1(t) h_{r-1}(t)( \rho - \|G\| ) \big( 1- \vp(t)^2 \|x_{r-1}(t)\|^2 \big) \\
& \quad + h_1(t) h_{r-1}(t) \|G\| \vp_1(t)^2 \big( - 2\rho\|x_{r-1}(t)\|(\kappa_1 - \rho  \kappa_{r-1}) + (\kappa_1 - \rho  \kappa_{r-1})^2 \big) \\
& > h_1(t) h_{r-1}(t)( \rho - \|G\| ) \big( 1- \vp(t)^2 \|x_{r-1}(t)\|^2 \big) \\
& \quad + h_1(t) h_{r-1}(t) \|G\| \vp_1(t)^2 \left( - 2 \rho \frac{1}{ \vp(t)} (\kappa_1 - \rho  \kappa_{r-1}) + (\kappa_1 - \rho  \kappa_{r-1})^2 \right) \\
& \overset{\rm \ref{Ass:funnel-functions}}{=} h_1(t) h_{r-1}(t)( \rho - \|G\| ) \big( 1- \vp(t)^2 \|x_{r-1}(t)\|^2 \big) \\
& \qquad + h_1(t) h_{r-1}(t) \|G\| \vp_1(t)^2 \left( -  \frac{2}{\vp_1(t)} (\kappa_1 - \rho  \kappa_{r-1}) + (\kappa_1 - \rho  \kappa_{r-1})^2 \right) \\
& = h_1(t) (\rho - \|G\| ) 
+  h_1(t) h_{r-1}(t) \|G\| \vp_1(t) \big(-2 (\kappa_1 - \rho  \kappa_{r-1}) + \vp_1(t) (\kappa_1 - \rho  \kappa_{r-1})^2 \big) \\
& > h_1(t) (\rho - \|G\| ) 
-2  h_1(t) h_{r-1}(t) \|G\| \vp_1(t) (\kappa_1 - \rho  \kappa_{r-1}) \big) \\
& = h_1(t) \Big( \rho - \|G\| - 2\vp_1(t) \|G\|h_{r-1}(t) ( \kappa_1 - \rho  \kappa_{r-1}) \Big) \\
& \overset{\eqref{eq:h_r-1-le}}{\ge}  h_1(t) \left( \rho - \|G\| - 2\vp_1(t) \|G\|\frac{1}{\rho \vp_1(t)  \kappa_{r-1}}( \kappa_1 - \rho  \kappa_{r-1}) \right) \\
& = h_1(t) \left( \rho + \|G\| - \frac{2 \|G\|}{\rho} \frac{\kappa_1}{\kappa_{r-1}} \right) \\
& \overset{\eqref{eq:kappas}}{=} h_1(t) \left( \rho + \|G\| - \frac{4 \rho^2 \|G\|}{\rho} \frac{ \kappa}{  \delta^{r-2} \hat \delta \kappa} \right) \\
& \overset{\eqref{def:delta},\eqref{def:hatdelta}}{>} h_1(t) \Big( \rho + \|G\| - 4 \|G\| \rho^2 (\rho+1)^{r-2} \Big) \\
& \overset{\eqref{def:Delta}}{=} h_1(t) \Delta_1 .
\end{align*}
\end{small}
%
Therefore, we find for~$t \in [t_0,t_1]$
\begin{equation} \label{eq:h1-h_r-1}
\rho h_1(t) - \|G\| h_{r-1}(t) > h_1(t) \Delta_1 \overset{\eqref{eq:relations-1-h-x}}{\ge} \frac{\lambda_1 \Delta_1}{2 \kappa} > 0 .
%
\end{equation}
To sum up, with~\eqref{eq:relations-1-h-x},~\eqref{eq:h_r-1-le} and~\eqref{eq:h1-h_r-1}, and the calculations above we obtain for~${t\in[t_0,t_1]}$
\begin{equation*}
\begin{aligned}
\ddt \tfrac{1}{2} \|x_1(t)\|^2 
& \overset{\eqref{eq:h1-h_r-1}}{<}  -\frac{\Delta_1 \lambda_1}{2 \kappa} \frac{\|x_1(t)\|^2}{\rho} + \left(\frac{\|G\| (\kappa_1 - \rho  \kappa_{r-1})}{\rho} h_{r-1}(t)  
+  \|b_1(t)\| \right) \|x_1(t)\| \\
& \overset{\eqref{eq:h_r-1-le}}{\le}  -\frac{\Delta_1  \lambda_1}{2 \rho \kappa}\|x_1(t)\|^2
+ \frac{ \|G\| }{\rho^2 \vp_1(t)  } \frac{\kappa_1}{\kappa_{r-1}} \|x_1(t)\| + \|x_1(t)\| \|b_1(t)\| \\
& \overset{\eqref{eq:kappas}}{=} -\frac{\Delta_1  \lambda_1}{2 \rho \kappa} \|x_1(t)\|^2
+ \frac{2 \|G\| }{ \vp_1(t)} \frac{1}{ \delta^{r-2} \hat \delta} \|x_1(t)\| + \|x_1(t)\| \|b_1(t)\| \\
& \overset{\eqref{def:delta}}{\le} -\frac{\Delta_1  \lambda_1}{2 \rho \kappa} \|x_1(t)\|^2 
+ \left( 2 \|G\| \sup_{s \in [\beta,\omega)} \psi_1(s) \rho (\rho+1)^{r-2} +  \|b_1(t)\| \right)\|x_1(t)\| \\
& \overset{\eqref{eq:relations-1-h-x}}{\le} 
\left( - \frac{\Delta_1 \lambda_1^2}{4 \rho \kappa} + 2 \|G\| \sup_{s \in [\beta,\omega)} \psi_1(s) \rho (\rho+1)^{r-2} + \sup_{s \in [0,\omega)} \|b_1(s)\| \right) \|x_1(t)\| .
\end{aligned}
\end{equation*}
Then, similar to~\eqref{eq:contradiction} we deduce a contradiction.
Therefore, in both cases~\eqref{eq:t0:psi1-x1-le-psi_r-1-x_r-1} and~\eqref{eq:t0:psi1-x1-g-psi_r-1-x_r-1}, and for $t \in [\beta, \omega)$ we have $\psi_1(t) - x_1(t) \ge \kappa_1$. 
Moreover, for~$t \in [0,\beta)$ we have $\psi_1(t) - \|x_1(t)\| \ge \kappa_1$ by definition of~$ \kappa$, and hence $h_1 \in \cL^{\infty}([0,\omega) \to \R)$.
Then successively we obtain $h_i \in \cL^\infty([0,\omega) \to \R)$ for all remaining $i \in \{2,\ldots,r-1\}$.

\ \\
If $\ell = 2$ we have, because~$t_0 = t_0^2 < t_0^1$ by~\eqref{eq:t0:at-least-one-smaller}
\begin{equation*}
\psi(t_0) - \|x_2(t_0)\| = \kappa_2  \overset{\eqref{eq:kappas}}{=}  \frac{\delta \kappa}{2 \rho^2} < \frac{\kappa_1}{\rho} \le  \frac{\psi_{1}(t_0)}{\rho} - \frac{\|x_{1}(t_0)\|}{\rho}
\end{equation*}
and, invoking the definition of~$t_0 = t_0^2$ there exists $ t_1 \in (t_0,t_0^1)$ such that for~${t \in [t_0,t_1]}$
\begin{equation} \label{eq:psi2-x-at-t0}
\begin{aligned}
\psi(t)& - \|x_2(t)\| \le \kappa_2 < \frac{\kappa_{1}}{\rho} \le \psi(t) - \frac{\|x_{1}(t)\|}{\rho} \\
\iff &  -\|x_2(t)\| \le \kappa_2 - \psi(t) < \frac{\kappa_{1}}{\rho} - \psi(t) \le  - \frac{\|x_{1}(t)\|}{\rho} \\
\iff &  0 \le \kappa_2 - \psi(t) + \|x_2(t)\| < \frac{\kappa_{1}}{\rho} - \psi(t) + \|x_2(t)\| \le \|x_2(t)\| - \frac{\|x_{1}(t)\|}{\rho} .
\end{aligned}
\end{equation}
So we readily conclude for $t \in [t_0,  t_1]$
\begin{subequations} \label{eq:relations-2-on-t0-hatt1}
\begin{equation} \label{eq:relations-2-x-h}
\begin{aligned}
\psi(t) - \| x_2(t)\| &\le \kappa_2 , \\
\| x_2 (t) \| & \ge \psi(t) - \kappa_2  \ge \frac{\lambda}{2}, \\
h_2(t) = \frac{1}{1-\vp(t)^2 \| x_2(t)\|^2} & \ge \frac{1}{2  \vp(t) \kappa_2 } \ge \frac{\lambda}{2 \kappa_2},
\end{aligned}
\end{equation}
and analogously to~\eqref{eq:h_r-1-le} we find 
\begin{equation} \label{eq:ell=2_h1}
h_1(t) 
= \frac{1}{1- \vp_1(t)^2 \|x_1(t)\|^2}
\le \frac{1}{\vp_1(t) \kappa_1},
\end{equation}
\end{subequations}
and furthermore, using~\eqref{eq:psi2-x-at-t0} and estimation~\eqref{eq:relations-2-x-h}, we have for~$t \in [t_0,t_1]$
\begin{equation} \label{eq:x_2-x_1}
\begin{aligned}
\|x_2(t)\| - \frac{\| x_{1}(t)\|}{\rho}  &\ge \|x_2(t)\| - \psi(t) + \frac{\kappa_{1}}{\rho} \\
&\ge \psi(t) - \kappa_2 - \psi(t) + \frac{\kappa_{1}}{\rho} 
\overset{\eqref{eq:kappas}}{=} \frac{\kappa}{\rho} -  \frac{\delta \kappa}{2 \rho^2} \overset{\eqref{def:delta}}{ >} 0.
\end{aligned}
\end{equation}
Note that~\eqref{eq:x_2-x_1} implies $\|x_1(t)\| < \rho \|x_2(t)\|$ for~$t \in [t_0, t_1]$.
We consider equation~\eqref{eq:ddt-w-bounded} for~$i=2$ and~$t \in [t_0, t_1]$ and obtain
\begin{equation} \label{eq:ddt-x2}
\begin{aligned}
\ddt \tfrac{1}{2} \| x_{2}(t)\|^2 &= -h_2(t) \| x_{2}(t)\|^2 + h_{1}(t) x_{2}(t)^\top x_{1}(t) + x_{2}(t)^\top b_2(t) \\
& \le -h_2(t) \| x_{2}(t)\|^2 +  h_{1}(t) \|x_{2}(t)\| \|x_{1}(t)\| + \|x_{2}(t)\| \|b_2(t)\| \\
& \overset{\eqref{eq:x_2-x_1}}{< } -h_2(t) \| x_{2}(t)\|^2 +  h_{1}(t) \rho \|x_{2}(t)\|^2  + \|x_{2}(t)\| \|b_2(t)\| \\
& = \Big(- h_2(t) + \rho h_{1}(t) \Big) \|x_2(t)\|^2    + \|x_{2}(t)\| \|b_2(t)\|.
\end{aligned}
\end{equation}
Thanks to~\eqref{def:delta} and~\eqref{eq:kappas} we have
\begin{equation} \label{eq:kappa1-2rhokappa2}
\kappa_1 - 2\rho^2 \kappa_2 = \kappa - 2 \rho^2 \frac{ \delta \kappa}{2\rho^2} = \kappa (1-\delta) \overset{\eqref{def:delta}}{>} 0. 
\end{equation}
Hence, using the property~\ref{Ass:funnel-functions} and the relations from~\eqref{eq:relations-2-on-t0-hatt1} we obtain for~$t \in [ t_0,  t_1]$ 
\begin{equation*}
\begin{aligned}
h_{2}(t)- \rho h_{1}(t) 
& \overset{\eqref{eq:relations-2-on-t0-hatt1}}{\ge}  \frac{1}{2  \vp(t) \kappa_2} - \frac{\rho}{\vp_1(t) \kappa_1}  
\overset{\rm \ref{Ass:funnel-functions}}{=} \frac{\kappa_1 - 2 \rho^2 \kappa_2}{2 \vp(t) \kappa_1 \kappa_2}  \\
&\overset{\eqref{eq:kappas},\eqref{eq:kappa1-2rhokappa2}}{=} \frac{2 \rho^2 (1-\delta) \kappa}{2 \vp(t) \delta \kappa^2 } 
\overset{\eqref{eq:relations-2-x-h}}{\ge} \frac{1-\delta}{\delta} \frac{ \rho^2 \lambda}{ \kappa} 
\overset{\eqref{def:delta}}{>} \frac{\rho^2 \lambda}{\kappa} .
\end{aligned}
\end{equation*}
With this, using~\eqref{eq:relations-2-x-h} we estimate~\eqref{eq:ddt-x2} for~$t \in [ t_0, t_1]$
\begin{equation*}
\begin{aligned}
\ddt \tfrac{1}{2} \| x_{2}(t)\|^2 & \overset{\eqref{eq:relations-2-x-h}}{< }\left( - \frac{ \rho^2 \lambda^2}{2 \kappa} + \sup_{s \in [0,\omega)}  \|b_2(s)\| \right) \|x_2(t)\|
\overset{\eqref{eq:L}}{\le} -  L \|x_2(t)\|.
\end{aligned}
\end{equation*}
%
%
Then, similar to~\eqref{eq:contradiction} a contradiction arises.
Therefore, for all $t \in [\beta, \omega)$ we have $\psi(t) - x_2(t) \ge \kappa_2$. 
Moreover, for~$t \in [0,\beta)$ we have $\psi(t) - \|x_2(t)\| \ge \kappa_2$ by definition of~$ \kappa$. Hence $h_2 \in \cL^{\infty}([0,\omega) \to \R)$.
Then successively we obtain $h_i \in \cL^\infty([0,\omega) \to \R)$ for all remaining ${i \in \{1,\ldots,r-1\}\setminus\{2\}}$.

\ \\
If $3 \le \ell \le r-1$ we have, because~$t_0 = t_0^{\ell} < t_0^{\ell-1}$ by~\eqref{eq:t0:at-least-one-smaller}
\begin{equation*}
 \psi(t_0) - \|x_\ell(t_0)\| = \kappa_\ell < \kappa_{\ell-1} \le  \psi(t_0) - \|x_{\ell-1}(t_0)\| .
\end{equation*}
Then, by invoking the definition of~$t_0 = t_0^{\ell}$ there exists $ t_1 \in (t_0,t_0^{\ell-1})$ such that 
\begin{equation} \label{eq:psi-x-at-t0}
\forall\, t \in [t_0,t_1]\,: \ \psi(t) - \|x_\ell(t)\| \le \kappa_\ell < \kappa_{\ell-1} \le \psi(t) - \|x_{\ell-1}(t)\|.
\end{equation}
As before, we deduce for~$t \in [t_0,  t_1]$
\begin{subequations} \label{eq:relations-3-on-t0-t1}
\begin{equation} \label{eq:relations-3-x-h}
\begin{aligned}
\psi(t) - \| x_\ell(t)\| &\le \kappa_\ell , \\
\| x_\ell (t) \| & \ge \psi(t) - \kappa_\ell  \ge \frac{\lambda}{2}, \\
h_\ell(t) = \frac{1}{1-\vp(t)^2 \| x_\ell(t)\|^2} & \ge \frac{1}{2 \vp(t) \kappa_\ell }, 
\end{aligned}
\end{equation}
similar to~\eqref{eq:h_r-1-le} we obtain
\begin{equation} \label{eq:ell-ge-3-h_ell-1-le}
h_{\ell-1}(t) 
= \frac{1}{1-\vp(t)^2 \|x_{\ell-1}(t)\|^2}
\le \frac{1}{\vp(t) \kappa_{\ell-1}},
\end{equation}
\end{subequations}
and, using~\eqref{eq:psi-x-at-t0} and~\eqref{eq:relations-3-x-h}, we have for~$t \in [t_0,  t_1]$
\begin{equation}
\begin{aligned} \label{eq:x_l-x_l-1}
\|x_\ell(t)\| - \| x_{\ell-1}(t)\|  &\ge \|x_\ell(t)\| - \psi(t) + \kappa_{\ell-1}  \\ 
 &\ge \psi(t) - \kappa_\ell - \psi(t) + \kappa_{\ell-1} 
 \overset{\eqref{eq:kappas}}{=} \frac{\delta^{\ell-2}}{2\rho^2} (1-\delta) \, \kappa > 0. 
\end{aligned}
\end{equation}
Note that~\eqref{eq:x_l-x_l-1} implies $\|x_\ell(t)\| > \|x_{\ell-1}(t)\| $ for~$t \in [t_0, t_1]$.
We consider equation~\eqref{eq:ddt-w-bounded} for~$i=\ell$ and~$t \in [t_0, t_1]$ and obtain
\begin{equation} \label{eq:ddt-xl}
\begin{aligned}
\ddt \tfrac{1}{2} \| x_{\ell}(t)\|^2 &= -h_\ell(t) \| x_{\ell}(t)\|^2 + h_{\ell-1}(t) x_{\ell}(t)^\top x_{\ell-1}(t) + x_{\ell}(t)^\top b_\ell(t) \\
& \le -h_\ell(t) \| x_{\ell}(t)\|^2 +  h_{\ell-1}(t) \|x_{\ell}(t)\| \|x_{\ell-1}(t)\| + \|x_{\ell}(t)\| \|b_\ell(t)\| \\
& \overset{\eqref{eq:x_l-x_l-1}}{<} \Big(- h_\ell(t) + h_{\ell-1}(t) \Big) \|x_\ell(t)\|^2    + \|x_{\ell}(t)\| \|b_\ell(t)\|.
\end{aligned}
\end{equation}
Recording
\begin{equation} \label{eq:kappa_ell}
\begin{aligned}
\kappa_{\ell-1} - 2 \kappa_{\ell} \overset{\eqref{eq:kappas}}{=} \frac{\delta^{\ell-2} \kappa }{2\rho^2} \left( 1- 2\delta \right)  
\overset{\eqref{def:Delta}}{=} \frac{\delta^{\ell-2} \Delta}{2 \rho^2} \, \kappa > 0 ,
\end{aligned}
\end{equation}
and using~\eqref{eq:relations-3-on-t0-t1} we estimate
\begin{equation*}
\begin{aligned}
h_{\ell}(t) - h_{\ell-1}(t)& 
\overset{\eqref{eq:relations-3-on-t0-t1}}{\ge} \frac{1}{2 \vp(t) \kappa_\ell} - \frac{1}{\vp(t) \kappa_{\ell-1}}  
= \frac{\kappa_{\ell-1} - 2 \kappa_\ell}{2 \vp(t) \kappa_{\ell-1} \kappa_\ell} \\
& \overset{\eqref{eq:kappas},\eqref{eq:kappa_ell}}{=} 
\frac{1}{2 \vp(t)}
\frac{\delta^{\ell-2} \Delta \kappa}{2 \rho^2} 
\frac{4 \rho^4}{\delta^{\ell-2} \delta^{\ell-1} \kappa^2} 
\overset{\eqref{def:delta},\eqref{eq:relations-3-x-h}}{>}  2^{\ell-1} \Delta \, \frac{\rho^2 \lambda}{\kappa}  .
\end{aligned}
\end{equation*}
With this and using~\eqref{eq:relations-3-x-h} we estimate~\eqref{eq:ddt-xl} for~$t \in [ t_0, t_1]$
\begin{equation*}
\begin{aligned}
\ddt \tfrac{1}{2} \| x_{\ell}(t)\|^2 &< \left( - 2^{\ell-1} \Delta \frac{\rho^2 \lambda^2}{ \kappa}  + \sup_{s \in [0,\omega)}  \|b_\ell(s)\| \right) \|x_\ell(t)\|
\overset{\eqref{eq:L}}{\le} -  L \|x_\ell(t)\|.
\end{aligned}
\end{equation*}
As before, a contradiction arises from analogous calculations as in~\eqref{eq:contradiction}.
Therefore, for all $t \in [\beta, \omega)$ we have $\psi(t) - \|x_\ell(t)\| \ge \kappa_\ell$.
Moreover, for~$t \in [0,\beta)$ we have ${\psi(t) - \|x_\ell(t)\| \ge \kappa_\ell}$ by definition of~$ \kappa$, and hence $h_\ell \in \cL^{\infty}([0,\omega) \to \R)$.
Then successively we obtain $h_i \in \cL^\infty([0,\omega) \to \R)$ for all remaining $i \in \{1,\ldots,r-1\}\setminus\{\ell\}$.

\ \\
\underline{Step 3}: 
We show $\omega = \infty$ and~$\tilde T \in \cT^{rm,\bar q}_\tau$.
Seeking a contradiction we assume~$\omega < \infty$. 
Then, since $h_i$ and $w_{i,j}$ for $i=1,\ldots,r-1$ and $j=1,\ldots,r$ are bounded via the previous steps, it follows that the graph of the solution of~\eqref{eq:ddt_w} is a compact subset of~$\cD$, a contradiction. Thus, ${\omega = \infty}$.
Therefore, the operator~$\tilde T$ is well defined and satisfies condition~\ref{T:BIBO} of Definition~\ref{Def:OP-T} by the previous calculations. 
Moreover, property~\ref{T:causal} of Definition~\ref{Def:OP-T} is satisfied. 
Further, note that~$\tilde T$ is defined via the solution of~\eqref{eq:ddt_w}, which depends linearly on $z,\dot z, \ldots, z^{(r-1)}$, $T \in \cT^{n,q}_{\tau,1}$, 
and since $f \in \cC(\R^r \times \R^{m(r-1)r} \to \R^m)$ its integral is locally Lipschitz continuous.
Therefore, the operator~$\tilde T$ satisfies~\ref{T:Lipschitz} of Definition~\ref{Def:OP-T} and hence~$\tilde T \in \cT^{rm,\bar q}_\tau$.

\ \\
Finally, we observe that the higher derivatives of~$z$ can be calculated via an successive application of the cascade's equations~\eqref{eq:FPC-cascade} and result in
\begin{equation} \label{eq:ddt_j-z}
z^{(j)}(t) = z_{r-1,j+1} + \sum_{k=0}^{j-1}\left( \frac{\text{\normalfont{d}}}{\text{\normalfont{d}} t} \right)^k \Big[ ( a_{r-k} + p_{r-k} h_{r-1}(t)) w_{r-1,1}(t) \Big],
\end{equation}
and hence, with~$z_{r-1,r+1}(t) := \tilde \Gamma u(t)$ the results above allow us to write the conjunction of~\eqref{eq:FPC-cascade} and~\eqref{eq:System} with input~$u$ and output~$z = z_{r-1,1}$ as
\begin{equation} \label{Def:tildeF}
\begin{aligned}
z^{(r)}(t) &= \sum_{k=0}^{r-1} \left( \frac{\text{\normalfont{d}}}{\text{\normalfont{d}} t} \right)^k \Big[ ( a_{r-k} + p_{r-k} h_{r-1}(t)) w_{r-1,1}(t) \Big] + \tilde \Gamma u(t) \\
&=: \tilde F\left( \tilde d(t), \tilde T(z,\dot z,\ldots,z^{(r-1)})(t) \right) + \tilde \Gamma u(t).
\end{aligned}
\end{equation}
where $\tilde F \in \cC(\R^r \times \R^{\bar q} \to \R^m)$,  
$\tilde d(t) := (\vp(t), \dot \vp(t), \ldots , \vp^{(r-1)}(t) )^\top \in \cL^{\infty}(\rp \to \R^r)$ and $\tilde T \in \cT^{rm, \bar q}_\tau$.
Therefore, $(\tilde d, \tilde F, \tilde T, \tilde \Gamma) \in \cN^{m,r}$ 
which completes the proof.
\qed
%


%

\end{document}